\documentclass[a4paper, reqno, 11pt]{amsart}
\usepackage[T1]{fontenc}
\usepackage[utf8]{inputenc}
\usepackage{hyperref}
\hypersetup{%
	colorlinks=true,
	urlcolor=blue,
	citecolor=purple,
	linkcolor=purple,
}

\usepackage{amssymb,amsthm,tikz-cd,stmaryrd, mathtools,amsmath}
\mathtoolsset{showonlyrefs}

\usepackage{rotating}

\tikzcdset{arrow style=tikz, diagrams={>=to}}
\usepackage{stackengine}

\usepackage{newpxtext}
\usepackage{newpxmath}

\usepackage[english]{babel}
\usepackage{csquotes}
\usepackage{graphicx}
\usepackage{overpic}

\usepackage{enumitem}
\setlist[itemize]{label=$\diamond$}

\makeatletter
\renewcommand{\paragraph}{%
	\@startsection{paragraph}{4}%
	{\z@}{.33em \@plus .5ex \@minus .2ex}{-1em}%
	{\normalfont\normalsize\bfseries}%
}
\makeatother

\newcommand{\Sp}{\operatorname{Sp}}
\newcommand{\Un}{\operatorname{U}}

\newcommand{\Fix}{\operatorname{Fix}}
\newcommand{\HF}{\operatorname{HF}}

\usepackage{dsfont}

\renewcommand{\phi}{\varphi}
\renewcommand{\epsilon}{\varepsilon}

\newcommand{\cA}{\mathcal{A}}

\newcommand{\bC}{\mathbb{C}}

\newcommand{\bN}{\mathbb{N}}

\newcommand{\bR}{\mathbb{R}}

\newcommand{\bZ}{\mathbb{Z}}

\newtheoremstyle{myprop}%
{1em plus .25em minus .3em}
{1em plus .25em minus .3em}
{\itshape}
{0em}
{\bfseries}
{ }
{5pt plus 1pt minus 1pt}
{}

\newtheoremstyle{mytheorem}%
{1em plus .25em minus .3em}
{1em plus .25em minus .3em}
{\itshape}
{0em}
{\bfseries}
{ }
{5pt plus 1pt minus 1pt}
{}

\newtheoremstyle{mydefinition}%
{1em plus .25em minus .3em}
{1em plus .25em minus .3em}
{\normalfont}
{0em}
{\bfseries}
{ }
{5pt plus 1pt minus 1pt}
{}

\newtheoremstyle{myremark}%
{1em plus .25em minus .3em}
{1em plus .25em minus .3em}
{\normalfont}
{0em}
{\bfseries\itshape}
{ }
{5pt plus 1pt minus 1pt}
{}

\theoremstyle{mydefinition}

\newtheorem*{defn*}{Definition}

\theoremstyle{myremark}
\newtheorem{rmk}{Remark}
\newtheorem*{rmk*}{Remark}

\theoremstyle{myprop}

\newtheorem{lemma}{Lemma}[section]
\newtheorem*{lemma*}{Lemma}
\newtheorem{exa}{Example}[section]

\theoremstyle{mytheorem}
\newtheorem{thm}{Theorem}
\newtheorem*{thm*}{Theorem}

\numberwithin{equation}{section}

\usepackage{amsthm}

\makeatletter
\newenvironment{myproof}
{\par\pushQED{\qed}% 压入证毕符号
	\normalfont\topsep6\p@\@plus6\p@\relax % 垂直间距
	\trivlist % 使用 trivlist 环境
	\item[\hskip\labelsep]\ignorespaces % 隐藏标题
}
{\popQED\endtrivlist\@endpefalse} % 弹出并放置证毕符号
\makeatother

\usepackage[backend=bibtex, style=numeric, giveninits=true, uniquename=init, sorting=nyt]{biblatex}
\usepackage{etoolbox}

\begin{document}
	
\title[Periodic Points of Hamiltonian Diffeomorphisms]{Periodic Points of Hamiltonian Diffeomorphisms Equal to Nondegenerate Linear Maps at Infinity}

  \date{\today}	
\author	{Meng Li}
\address{	School of Mathematics, Shandong University\\
	Jinan, Shandong 250100\\
	The People's Republic of China}

 \email{mengli1@mail.sdu.edu.cn}
 \thanks{This work is supported by NSFC($\sharp$12371192).}

\maketitle

\begin{abstract}
	We study Hamiltonian diffeomorphisms on symplectic Euclidean spaces that are equal to non-degenerate linear maps at infinity. Under the assumption that there exists an isolated homologically nontrivial fixed point satisfying the twist condition, we prove the existence of infinitely many simple periodic points. More precisely, if such a diffeomorphism has only finitely many fixed points, then it admits simple periodic points with arbitrarily large prime periods.
\end{abstract}

\section{Introduction}	

The classical Poincar\'e--Birkhoff theorem establishes the existence of at least two fixed points for area-preserving homeomorphisms of the planar annulus that twist the boundary circles in opposite directions. This foundational result has inspired extensive work on forced oscillations in Hamiltonian systems, leading to profound developments in symplectic topology.

A significant generalization on was achieved by \citeauthor{FU2017} \cite{FU2017}, who replaced the boundary preservation condition with a difference in rotation angles for Hamiltonian diffeomorphisms on $\mathbb{R}^{2n}$, obtaining at least $n+1$ fixed points. Building on this, \citeauthor{BAME2023} \cite{BAME2023} studied planar systems linearizable at both the origin and infinity, providing a rigorous analysis of the relationship between rotation angles in linear Hamiltonian systems and Conley–Zehnder indices. They showed that a difference in mean indices—an analogue of the classical twist condition—implies the existence of simple periodic points of arbitrarily large period. It is worth noting that the linear systems at infinity and at the origin are not required to be non-degenerate. 

Extending these ideas to higher dimensions presents substantial challenges, as the relationship between rotation angles and Conley--Zehnder indices becomes intractable. Nevertheless, Floer-theoretic approaches have yielded significant progress.  Under different assumptions,
\citeauthor{GB2014} \cite{GB2014} and \citeauthor{ML2024} \cite{ML2024} have independently studied the high-dimensional case. The Hamiltonian function $H_t(z)$ studied by \citeauthor{GB2014} coincides with an autonomous, non-degenerate quadratic form $Q$ outside a compact set, and the associated linear Hamiltonian vector field $X_Q$ is required to have only real eigenvalues or complex eigenvalues $\sigma$ satisfying $|\mathrm{Re}\,\sigma| > |\mathrm{Im}\,\sigma|$. \citeauthor{ML2024} considers non-autonomous quadratic forms at infinity whose time-one maps are unitary matrices. Despite these technical differences, both works establish that if a Hamiltonian diffeomorphism possesses a non-degenerate fixed point whose mean  index differs from that at infinity, then infinitely many  periodic points exist. This conclusion holds even for isolated, homologically nontrivial fixed points.

Building upon these advancements, within the framework of Floer homology, we carry out a more refined reduction of symplectic matrices and quadratic forms. This allows us to weaken the assumptions on the quadratic form at infinity and establish existence results under significantly more general conditions.

Moreover, analogous problems can be studied in the setting of Liouville domains. Under suitable twisting conditions, the existence of infinitely many periodic points can be established \cite{MAKO2022}.

\subsection*{Main Results}

We consider smooth Hamiltonians $H_t(z) \in C^\infty(S^1 \times \mathbb{R}^{2n})$ that equal to non-degenerate quadratic forms $Q_t(z)$ at infinity. Specifically,
\begin{equation}
	H_t(z) = Q_t(z) + h_t(z) = \frac{1}{2} \langle B_t z, z \rangle + h_t(z),
\end{equation}
where $B_t$ is a real symmetric matrix and $h_t(z)$ has compact support: $h_t(z) \equiv 0$ for $|z| \geq R_0$ with some $R_0 > 0$. Let$\phi^t_H$ denote the flow generated by the Hamiltonian vector field $X_H^t$. The quadratic form \( Q_t \) is said to be \emph{non-degenerate} if the linear map \( \phi^1_Q \) does not have 1 as an eigenvalue.

For a fixed point $z_0$ of $\phi^1_H$ or a quadratic form $Q_t(z)$, we define the Conley--Zehnder indices $i_H(z_0)$, $i_\infty(H)$ and their mean indices $\hat{i}_H(z_0)$, $\hat{i}_\infty(H)$. A fixed point $z_0$ is a \emph{twist fixed point} if $\hat{i}_H(z_0) \neq \hat{i}_\infty(H)$. An isolated fixed point is \emph{homologically nontrivial} if its local Floer homology is non-zero. In particular, a non-degenerate fixed point $z_0$ is necessarily homologically nontrivial. Homological nontriviality can be equivalently characterized in terms of generating functions; see \cite{AV2013}. A periodic point is \emph{simple} if it is not an iteration of a periodic point with smaller period. Finally denote by $\Fix(\phi^1_H)$ the collection of fixed points of $\phi^1_H$.

\begin{thm}\label{theo1}
	Let $H: S^1 \times \mathbb{R}^{2n} \to \mathbb{R}$ be a smooth Hamiltonian that equal to a non-degenerate quadratic form $Q_t(z)$ at infinity. If $\phi^1_H$ has an isolated, homologically nontrivial fixed point $z_0$ satisfying the twist condition $\hat{i}_H(z_0) \neq \hat{i}_\infty(H)$, and if $\Fix(\phi_H^1)$ is finite, then $\phi^1_H$ possesses simple periodic points with arbitrarily large prime periods.
\end{thm}

Unlike previous works \cite{GB2014,ML2024}, Theorem \ref{theo1} imposes no additional conditions on the quadratic form $Q_t$ beyond non-degeneracy. This generality enables applications to systems of the form  \begin{equation}
	u'' + \nabla_u F(t,u) = 0,
\end{equation}
where $F \in C^\infty(S^1 \times \mathbb{R}^N)$ satisfies $F(t,0) = 0$ and admits symmetric matrices $A_0(t)$, $A_\infty(t)$ such that
\begin{equation}
	\lim_{|u|\to 0} \frac{|\nabla_u F(t,u)|}{|u|} = A_0(t), \quad \frac{|\nabla_u F(t,u)|}{|u|} = A_\infty(t) \text{ for } |u| \geq R_0 .
\end{equation}
 When the associated linear Hamiltonian systems
\begin{equation}
	\dot{\begin{pmatrix} q \\ p \end{pmatrix}} = 
	\begin{pmatrix}
		0 & I_N \\ -A_0(t) & 0
	\end{pmatrix}
	\begin{pmatrix} q \\ p \end{pmatrix}
	\quad \text{and} \quad
	\dot{\begin{pmatrix} q \\ p \end{pmatrix}} = 
	\begin{pmatrix}
		0 & I_N \\ -A_\infty(t) & 0
	\end{pmatrix}
	\begin{pmatrix} q \\ p \end{pmatrix}
\end{equation}
 are non-degenerate and have distinct mean indices, the equation admits infinitely many simple periodic solutions. 
While stronger one-dimensional results exist via the Poincar\'e--Birkhoff theorem \cite{BAOR2014}, higher-dimensional analogues typically require additional structure.

Refined index analysis yields further consequences in low dimensions:

\begin{thm}\label{theo2}
	Let $H: S^1 \times \mathbb{R}^{2n} \to \mathbb{R}$ be a smooth Hamiltonian that equal to a non-degenerate quadratic form $Q_t(z)$ at infinity, with $\phi^1_H$ non-degenerate, $\Fix(\phi^1_H)$ finite, and at least two fixed points.
	\begin{itemize}
		\item For $n = 1$, $\phi^1_H$ has simple periodic orbits with arbitrarily large prime periods.
		\item For $n = 2$, if all eigenvalues of $\phi^1_Q$ are entirely positive , entirely negative, or form a quadruple $\{\rho \omega, \rho \overline{\omega}, \rho^{-1}\omega, \rho^{-1}\overline{\omega}\} \subset \mathbb{C} \setminus (\mathrm{U} \cup \mathbb{R})$, then $\phi^1_H$ has simple periodic orbits with arbitrarily large prime periods.
	\end{itemize}
\end{thm}

The study of periodic orbits in Hamiltonian systems has evolved through several key developments. \citeauthor{Abb2001} \cite{Abb2001} established that  for two-dimensional asymptotically linear Hamiltonians with non-degenerate quadratic forms at infinity and non-degenerate $\phi^1_H$ admit infinitely many simple periodic points when at least two fixed points exist. Subsequent work by \citeauthor{GB2014} \cite{GB2014} removed the non-degeneracy condition on $\phi^1_H$, showing that for Hamiltonians equal to hyperbolic quadratic forms at infinity, it suffices to have at least two isolated homologically nontrivial fixed points. In the elliptic case, Franks' theorem provides a stronger conclusion: $\phi^1_H$ must have either exactly two or infinitely many periodic points, without homological conditions.

This progression naturally extends to higher dimensions. Abbondandolo \cite{Abb2001} conjectured that for asymptotically linear Hamiltonian systems in arbitrary dimensions, the presence of at least two fixed points under suitable non-degeneracy conditions implies infinitely many simple periodic points. This open conjecture shares profound connections with the Hofer-Zehnder conjecture \cite{HHZE1995} for compact symplectic manifolds, which asserts that Hamiltonian diffeomorphisms with more fixed points than the Arnold-conjectured minimum necessarily possess infinitely many periodic orbits. Our Theorem  \ref{theo2}  significantly extends these partial four-dimensional results \cite{GB2014}—previously limited to autonomous  hyperbolic systems $X_Q$ with exclusively real eigenvalues—to encompass a substantially broader class of eigenvalue configurations, though numerous challenging cases remain unresolved.

The foundational work of Conley and Zehnder \cite{conley1984} revealed that asymptotically linear Hamiltonian systems on $\mathbb{R}^{2n}$ with non-degenerate quadratic forms at infinity always yield at least one fixed point. This suggests that even a single "excess" fixed point—beyond the topologically guaranteed minimum—should force the emergence of infinitely many  periodic ponits. Our Theorem \ref{theo1} establishes a result for Hamiltonians that coincide with a non-degenerate quadratic forms at infinity, but it remains far from fully resolving Abbondandolo's conjecture.

\subsection*{Structure of the Paper}

Section \ref{Preliminaries} introduces our conventions and the foundational aspects of Floer homology on $\mathbb{R}^{2n}$, along with the normal forms of symplectic matrices that characterize quadratic behavior at infinity. Section \ref{Construction and Properties of the Functions} constructs key functions and establishes their properties, paving the way for the proof of Theorem \ref{theo1}. Finally, Section \ref{The proof of Theorem} verifies the well-definedness of Floer homology for our constructed functions and presents the proofs of Theorems \ref{theo1} and \ref{theo2}.

\section{Preliminaries}
\label{Preliminaries}

\subsection{Conventions and notation}
\label{Conventions and notation}
In this paper, we equip $\bR^{2n}$  with the coordinates $(q_1,\ldots,q_n,p_1,\ldots, p_n)$, the standard Liouville form $\lambda_0 = \sum_{j=1}^n p_j\,dq_j$, and the standard symplectic form 
\[
\omega_0 = d\lambda_0 = \sum_{j=1}^n dp_j \wedge dq_j.
\]
The linear automorphism $J_0: \mathbb{R}^{2n} \to \mathbb{R}^{2n}$, defined by $(q,p) \mapsto (-p,q)$, is the standard complex structure. These structures are related by
\[
\omega_0(u,v) = -J_0 u \cdot v,
\]
and the metric induced by $\omega_0(J_0 u, v)$ coincides with the standard Euclidean metric.

For a Hamiltonian $H \in C^\infty(S^1 \times \mathbb{R}^{2n})$, the associated Hamiltonian vector field $X_H^t(z)$ is defined by $i_{X_H^t(z)} \omega_0 = -dH$, or equivalently by the Hamiltonian system
\begin{equation}\label{eq:HSy}
	\dot{z} = X_H^t(z) = -J_0 \nabla H.
\end{equation}
We denote the time-dependent flow of $X_H$ by $\phi_H^t$. One-periodic (resp. $k$-periodic) solutions of $X_H$ correspond bijectively to fixed points (resp. $k$-periodic points) of $\phi_H^1$.
For a loop $\phi_H^t(z_0) = x(t): S^1 \to \mathbb{R}^{2n}$, define the action functional by
\[
\mathcal{A}_H(x) = \frac{1}{2} \int_{S^1} J_0 \dot{x} \cdot x \, dt - \int_{S^1} H_t(x(t)) \, dt.
\]
We also write $\mathcal{A}_H(z_0)$ for $\mathcal{A}_H(x)$. The critical points of $\mathcal{A}_H$ are precisely the one-periodic solutions of $X_H$; we denote this set by $\mathcal{P}_H$. Let $\mathcal{P}_H^a \subset \mathcal{P}_H$ be the subset with action $\mathcal{A}_H(x) < a$, and $\mathcal{P}_H^{[a,b)} = \mathcal{P}_H^b / \mathcal{P}_H^a$ the subset with action in $[a,b)$. The action spectrum $\mathscr{L}(H)$ of $H_t$ is the set of critical values of $\mathcal{A}_H$; it is a closed set of measure zero \cite{HHZE1995, SM2000}.

If $z_0$ is a fixed point of $\phi^1_H$, or equivalently, if $\phi^t_H(z_0)=\overline{z}(t)$ is a one-periodic solution of the Hamiltonian vector field $X_H$, we can linearize the system \eqref{eq:HSy} along $\overline{z}(t)$ to obtain
\begin{equation}\label{eq:LinSys}
	\dot{z} = -J_0 \frac{\partial^2}{\partial z^2} H(t, \overline{z}(t)) z.
\end{equation}
Let $\gamma(t)$ denote the fundamental solution matrix of \eqref{eq:LinSys}, which satisfies
\begin{equation}\label{diedai}
	\gamma(t) = \gamma(t-j) \gamma(1)^j, \quad \forall j \le t \le j+1, \, j \in \bN.
\end{equation}
We say that the fixed point $z_0$ of $\phi^1_H$ (or the one-periodic solution $\overline{z}$) is non-degenerate if $\gamma(1)$ does not have $1$ as an eigenvalue. Otherwise, $z_0$ or $\overline{z}$ is called degenerate. The Hamiltonian $H_t(z)$ is said to be nondegenerate if all its one-periodic solutions are nondegenerate. We call a positive integer $k$ is \emph{admissible with respect to $z_0$ or $\overline{z}(t)$} if $\lambda^k \ne 1$ for every eigenvalue $\lambda \ne 1$ of $\gamma(1)$. Moreover, $k$ is said to be admissible for $\phi^1_H$ if it is admissible with respect to every fixed point of $\phi^1_H$. If $z_0$ is an isolated fixed point of $\phi^1_H$, then for every admissible $k$, $z_0$ is also an isolated fixed point of $\phi^{k}_H$ (see \cite{GVBZ2010}).

For any $s \in \bN^+$, the  index of the symplectic path $\gamma(t)|_{[0,s]}$ is defined as an integer, whether or not the eigenvalues of $\gamma(s)$ contain $1$, denoted by $i_H(z_0, s)$ or $i_H(\overline{z}, s)$. For details, we refer to \cite{LY2012}. When $s=1$, this index is simply the Conley–Zehnder index, denoted by $i_H(z_0)$ or $i_H(\overline{z})$. The mean index of $z_0$ or $\overline{z}$ is defined as the mean index of the path $\gamma(t)$, denoted by $\hat{i}_H(z_0)$ or $\hat{i}_H(\overline{z})$, and is given by
\begin{equation}
	\hat{i}_H(z_0) = \hat{i}_H(\overline{z}) = \lim_{s \to +\infty} \frac{i_H(z_0, s)}{s}.
\end{equation}

Now assume that the Hamiltonian function $H_t$ coincides with a non-degenerate quadratic form $Q_t$ at infinity. A positive integer $k$ is \emph{admissible with respect to $Q_t$} if $\lambda^k \ne 1$ for every eigenvalue $\lambda \ne 1$ of $\phi^1_Q$. Note that if $Q_t$ is nondegenerate, then so is $k Q_{kt}$. The Hamiltonian flow $\phi^t_Q$ is a symplectic matrix path satisfying \eqref{diedai}. The  index of $\phi^t_Q|_{[0,s]}$ and the mean index of $\phi^t_Q$ are defined and denoted by $i_{\infty}(H, s)$ and $\hat{i}_{\infty}(H)$, respectively. In particular, when $s=1$, this index is the Conley–Zehnder index, denoted by $i_{\infty}(H)$.

The Conley-Zehnder index defined in \cite{LY2012}, \cite{SZ1992}, and \cite{GJ2012} counts half-turns in the counterclockwise direction for certain eigenvalues. In this paper, we adopt the opposite convention: the indices $i_H(\overline{z},s)$ for a one-periodic solution $\overline{z}$ and $i_{\infty}(H,s)$ for a non-degenerate quadratic form at infinity are defined as the negatives of those in \cite{LY2012}. That is, our Conley-Zehnder index counts half-turns in the clockwise direction.

This normalization is chosen so that $i_H(\overline{z}) = n$ when $\overline{z}$ is a non-degenerate maximum of an autonomous Hamiltonian $H$ with small Hessian. More generally, if $S$ is an invertible $2n \times 2n$ symmetric matrix with $\|S\| < 2\pi$ and $\psi(t) = e^{tJ_0S}$ is the corresponding symplectic path, then the Conley-Zehnder index is given by
\begin{equation}
	i(\psi) = \mathrm{Ind}(S) - n,
\end{equation}
where $\mathrm{Ind}(S)$ denotes the number of negative eigenvalues of $S$.

\subsection{Floer homology  }
\label{Floer homology }

Let $J$ be a smooth almost complex structure on $\mathbb{R}^{2n}$ that may depend on $(s,t,z) \in \mathbb{R} \times S^1 \times \mathbb{R}^{2n}$ and is $\omega_0$-compatible, meaning that
\[
g_J(u,v) = \omega_0(Ju,v)
\]
defines an $(s,t)$-dependent family of Riemannian metrics. Assume $J$ is uniformly bounded as an endomorphism of $\mathbb{R}^{2n}$. Then the metrics $g_J$ are uniformly equivalent to the Euclidean metric $g_{J_0}(u,v) = u \cdot v$. Denote the associated norms by $|\cdot|_J$, and let $\nabla_J$ be the gradient operator with respect to $g_J$. With our sign conventions, the Hamiltonian vector field satisfies $X_H = -J \nabla_J H$.

A Floer trajectory is a map $u: \mathbb{R} \times S^1 \to \mathbb{R}^{2n}$ satisfying the Floer equation
\begin{equation}\label{eq:Floer eq}
	\partial_s u + J(s,t,u)(\partial_t u - X_H(u)) = 0.
\end{equation}
For such a solution $u$, its energy is defined as
\[
E(u) = \int_{\mathbb{R} \times S^1} |\partial_s u|_J^2 \, ds \, dt.
\]
The following theorem provides uniform bounds for solutions of the Floer equation.

\subsubsection{The well-defined of Floer homology and Filtered Floer homology } 
\label{The well-defined  of Floer homology and Filtered Floer homology } 

we will consider the following growth assumptions on Hamiltonian functions $H\in C^{\infty}(S^1\times \bR^{2n})$:
\begin{enumerate}[label=\textbf{(H\arabic*)}]
	\item \label{Linear} \emph{Linear growth of the Hamiltonian vector field}. The Hamiltonian vector field $X_H$ is said to have \emph{linear growth at infinity} if there exists a positive number $c$ such that $|X_{H}(z)|\le c(1+|z|)$ for every $(t,z)\in S^1\times \bR^{2n}$.
	\item \label{nonres} \emph{Nonresonance at infinity}. The Hamiltonian $H_t$ is said to be \emph{nonresonance at infinity} if there exist positive number $\epsilon>0$ and $r>0$ such that for every smooth curve $z:S^1\rightarrow \bR^{2n}$ satisfying
	\begin{equation}
		\parallel \dot{z}-X_H(z)\parallel _{L^2(S^1)} \le \epsilon,
	\end{equation}       
	there holds  $\parallel z \parallel _{L^2(S^1)} \le r$.
\end{enumerate} 

\begin{thm} [\cite{AAJ2022}]\label{The:energy bounds}
 Let $J$ be a  uniformly bounded $\omega_0$-compatible almost complex structure on $\bR^{2n}$, smoothly depending on $(s,t,z)\in \bR \times S^1 \times \bR^{2n}$, and let $H\in C^{\infty}(S^1\times \bR^{2n})$ be a smooth Hamiltonian that satisfies conditions \ref{Linear} and \ref{nonres}.  For evey $E>0$, there is a positive number $M=M(E)$ such that every solution $u\in C^{\infty}(\bR\times S^1, \bR^{2n})$ of the Floer equation \eqref{eq:Floer eq} with energy bound
	\begin{equation}
		E(u)=\int_{\bR\times S^1} |\partial_su|_J^2dsdt\le E
	\end{equation}	
	satisfies
	\begin{equation}
		\sup\limits_{(s,t)\in \bR\times S^1} |u(s,t)|\le M.
	\end{equation}
\end{thm}

\begin{rmk}\label{Hscon}
 With a slight modification of the above theorem, we can extend the above result to $s$-dependent Hamiltonian functions.
	Similarly, if the Hamiltonian function $H^s_t(z)\in C^{\infty}(\bR\times S^1\times \bR^{2n})$ satisfies the following two conditions , the conclusion of this theorem can also be obtained(\cite{AAJ2022}).
\begin{enumerate}[label=\textbf{(H\arabic*')}]
	\item \label{Linearinf} $H^s_t(z)$ depends on $s$ only for $s$ in a bounded interval, that for $s$ outside of this interval $H^s_t(z)$ is nonresonant at infinity
	\item \label{nonresuni} the Hamiltonian vector field of $H^s$ has linear growth at infinity, uniformly on $s\in \bR$.
	
\end{enumerate} 	

\end{rmk}

When the Hamiltonian functions $H$ and $H^s$ satisfy the aforementioned conditions, Floer homology is well-defined. The energy bound on Floer trajectories implies an $L^\infty$ bound. On compact manifolds, such bounds are automatic, but in the non-compact setting of $\mathbb{R}^{2n}$, additional conditions are required. By adapting arguments from the compact case, one can show that the gradients of Floer trajectories are also uniformly bounded. The Ascoli-Arzelà theorem then yields compactness of the set of Floer trajectories with bounded energy in $C^\infty_{\text{loc}}(\mathbb{R} \times S^1, \mathbb{R}^{2n})$; see \cite{AMDM2014} for details.

Now assume $H_t(z)$ is non-degenerate and satisfies conditions \ref{Linear} and \ref{nonres}, and let $J$ be a uniformly bounded, $\omega_0$-compatible almost complex structure. For any two one-periodic solutions $x(t)$ and $y(t)$ of $X_H$, define the space of Floer trajectories connecting them as:
\begin{equation}
	\mathcal{M}_H(x,y,J) = \left\{ u: \mathbb{R} \times S^1 \to \mathbb{R}^{2n} \,\middle|\, 
	\begin{array}{l}
		u \in C^\infty \text{ solves \eqref{eq:Floer eq} with finite energy}, \\
		\lim\limits_{s \to -\infty} u(s,\cdot) = x, \quad 
		\lim\limits_{s \to +\infty} u(s,\cdot) = y
	\end{array} \right\}.
\end{equation}
For any $u \in \mathcal{M}_H(x,y,J)$, the energy identity holds:
\begin{equation}
	E(u) = \mathcal{A}_H(x) - \mathcal{A}_H(y).
\end{equation}
The space $\mathcal{M}_H(x,y,J)$ carries a natural $\mathbb{R}$-action via $(\tau \cdot u)(s,t) = u(s+\tau,t)$, and we denote the quotient by $\widehat{\mathcal{M}}_H(x,y,J) = \mathcal{M}_H(x,y,J)/\mathbb{R}$.

More importantly, the transversality property ensures that, after a small perturbation of $J$, for any two distinct one-periodic solutions $x \neq y$, the space $\mathcal{M}_H(x,y,J)$ becomes a smooth manifold of dimension $i_H(x) - i_H(y)$. A pair $(H,J)$ is called \emph{regular} if it satisfies this transversality condition, and we denote by $(\mathcal{H},\mathcal{J})_{\text{reg}}$ the set of all such regular pairs.

Let $CF_k(H,J)$ be the $\mathbb{Z}/2$-vector space generated by one-periodic solutions of $X_H$ with Conley-Zehnder index $k$. The differential is defined as
\begin{equation}
	\partial: CF_k(H,J) \to CF_{k-1}(H,J), \quad \partial(x) = \sum_y n(x,y) y,
\end{equation}
where $x$ is a one-periodic solution with $i_H(x) = k$, the sum ranges over all one-periodic solutions $y$ with $i_H(y) = k-1$, and $n(x,y)$ counts (modulo 2) the number of points in $\widehat{\mathcal{M}}_H(x,y,J)$. Theorem \ref{The:energy bounds} ensures uniform $L^\infty$-bounds, which imply that $\widehat{\mathcal{M}}_H(x,y,J)$ is compact in $C^\infty_{\text{loc}}(\mathbb{R} \times S^1, \mathbb{R}^{2n})$. When this space is zero-dimensional, it is a finite set, making $n(x,y)$ well-defined. For further details, see \cite{AAJ2022} and \cite{AMDM2014}.

Since $\partial \circ \partial = 0$, we define the Floer homology $HF_*(H,J)$ as the homology of the complex $CF_*(H,J)$. Although $HF_*(H,J)$ depends on $H$, it is independent of $J$ within regular pairs, so we denote it simply by $HF_*(H)$.

Now let $a < b$ be real numbers outside $\mathscr{L}(H)$ (the set of critical values of $\mathcal{A}_H$). Let $CF_*^a(H,J)$ be the subspace generated by one-periodic solutions with $\mathcal{A}_H(x) < a$. The filtered Floer complex for $[a,b)$ is defined as
\[
CF_k^{[a,b)}(H,J)= CF_k^b(H,J) / CF_k^a(H,J),
\]
a $\mathbb{Z}_2$-vector space generated by $x \in \mathcal{P}_H^{[a,b)}$ with $i_H(x) = k$. Its differential is
\begin{equation}
	\partial: CF_k^{[a,b)}(H,J) \to CF_{k-1}^{[a,b)}(H,J), \quad \partial(x) = \sum_y n(x,y) y,
\end{equation}
where $x \in \mathcal{P}_H^{[a,b)}$ with $i_H(x) = k$, and the sum is over $y \in \mathcal{P}_H^{[a,b)}$ with $i_H(y) = k-1$. The resulting homology $HF_*^{[a,b)}(H)$ is called the filtered Floer homology.

These constructions extend to degenerate Hamiltonians $H_t(z)$ via $C^2$-small perturbations $\tilde{H}$ that yield non-degenerate one-periodic solutions. We define $HF_*^{[a,b)}(H) = HF_*^{[a,b)}(\tilde{H})$. If $a,b \notin \mathscr{L}(H)$, then for sufficiently small perturbations, $a,b \notin \mathscr{L}(\tilde{H})$ as well. The groups $HF_*^{[a,b)}(H)$ are canonically isomorphic for different choices of $\tilde{H}$ near $H$, and all results discussed here hold for $HF_*^{[a,b)}(H)$; see \cite{BPLP2003}, \cite{FAHH1994}, \cite{SM2000}.

\subsubsection{$C$-bounded homotopy $H^s$}
\label{$C$-bounded homotopy $H^s$}
Let $(H_0,J_0)$ and $(H_1,J_1)$ be regular pairs in $(\mathcal{H},\mathcal{J})_{\text{reg}}$, and let $(H^s,J^s)$ be a smooth homotopy connecting them such that:
\[
\begin{cases}
	H^s = H_0 & \text{for } s \le -\kappa_0, \\
	H^s = H_1 & \text{for } s \ge \kappa_0,
\end{cases}
\quad
\begin{cases}
	J^s = J_0 & \text{for } s \le -\kappa_0, \\
	J^s = J_1 & \text{for } s \ge \kappa_0,
\end{cases}
\]
for some constant $\kappa_0 > 0$. After a small perturbation, we may assume $(H^s,J^s)$ is regular.

A homotopy $H^s$ is called $C$-bounded ($C \in \mathbb{R}$) if
\[
\int_{-\infty}^{\infty} \int_{S^1} \max_{z \in \mathbb{R}^{2n}} \partial_s H^s_t(z) \, dt \, ds \le C.
\]
Note that any $C$-bounded homotopy is also $C'$-bounded for $C' \ge C$; in what follows we assume $C \ge 0$.

Assume $H^s$ satisfies conditions \ref{Linearinf} and \ref{nonresuni}, is $C$-bounded, and $J^s$ is uniformly bounded and $\omega_0$-compatible. For $x \in \mathcal{P}_{H_0}$ and $y \in \mathcal{P}_{H_1}$, let $\mathcal{M}_{H^s}(x,y,J^s)$ denote the space of Floer trajectories connecting $x$ to $y$ with respect to $(H^s,J^s)$. For any $u \in \mathcal{M}_{H^s}(x,y,J^s)$, the energy identity holds:
\[
E(u) = \mathcal{A}_{H_0}(x) - \mathcal{A}_{H_1}(y) + \int_{-\infty}^{\infty} \int_{S^1} \partial_s H^s_t(u) \, dt \, ds,
\]
which implies $E(u) \le C + \mathcal{A}_{H_0}(x) - \mathcal{A}_{H_1}(y)$.

Now fix $a < b$ with $a,b \notin \mathscr{L}(H_0)$. For $x \in \mathcal{P}_{H_0}^{[a,b)}$ and $y \in \mathcal{P}_{H_1}^{[a+C,b+C)}$, the regularity of $(H^s,J^s)$ ensures $\mathcal{M}_{H^s}(x,y,J^s)$ is a smooth manifold of dimension $i_{H_0}(x) - i_{H_1}(y)$. By Remark \ref{Hscon}, conditions \ref{Linearinf} and \ref{nonresuni} imply uniform $L^\infty$-bounds, so $\mathcal{M}_{H^s}(x,y,J^s)$ is compact in $C^\infty_{\text{loc}}(\mathbb{R} \times S^1, \mathbb{R}^{2n})$. In particular, when $i_{H_0}(x) = i_{H_1}(y)$, it is a finite set.
Following \cite{GVL2007}, we define a chain map
\[
\Psi_{H_0,H_1}: CF_k^{[a,b)}(H_0,J_0) \to CF_k^{[a+C,b+C)}(H_1,J_1), \quad \Psi_{H_0,H_1}(x) = \sum_y n(x,y) y,
\]
where the sum ranges over $y \in \mathcal{P}_{H_1}^{[a+C,b+C)}$ with $i_{H_1}(y) = i_{H_0}(x)$, and $n(x,y)$ counts (modulo 2) the points in $\mathcal{M}_{H^s}(x,y,J^s)$. This induces a map on filtered Floer homology:
\[
\Psi_{H_0,H_1}: HF^{[a,b)}(H_0) \to HF^{[a+C,b+C)}(H_1).
\]

\subsubsection{Local Floer homology}
\label{Local Floer homology}

Let $\gamma$ be an isolated one-periodic solution of the Hamiltonian vector field $X_H$. Choose a sufficiently small tubular neighborhood $U$ of $\gamma$, and consider a nondegenerate $C^2$-small perturbation $\widetilde{H}$ of $H$ supported in $U$ such that all one-periodic solutions of $\widetilde{H}$ within $U$ are nondegenerate. Such perturbations exist; see \cite{SZ1992}. Moreover, if $\|\widetilde{H} - H\|_{C^2}$ and $\mathrm{supp}(\widetilde{H} - H)$ are sufficiently small, then every Floer trajectory $u$ connecting two such solutions remains in $U$; see \cite{SDM1990}, \cite{SD1997}. After a small perturbation of the almost complex structure to achieve transversality, the $\mathbb{Z}_2$-vector space generated by the one-periodic solutions of $\widetilde{H}$ in $U$ forms a chain complex with the standard Floer differential. A continuation argument shows that the homology of this complex is independent of $\widetilde{H}$ and the almost complex structure \cite{SZ1992}. We call the resulting homology group $HF^{\text{loc}}_*(H,\gamma)$ the \emph{local Floer homology}. Such groups were first considered by Floer \cite{FA11989, FA21989}; for their definition and properties, see \cite{GVL2010}.

\begin{exa}
	If $\gamma$ is nondegenerate with Conley-Zehnder $i_H(\gamma) = k$, then $HF^{\text{loc}}_*(H,\gamma) = \mathbb{Z}_2$ for $* = k$ and $0$ otherwise.
\end{exa}

\begin{lemma}\label{HFre}
	Assume all one-periodic solutions of $X_H$ lie in a compact set, and let $c \in \mathbb{R}$ be such that every one-periodic solution $\gamma_i$ of $X_H$ with action $c$ is isolated. Then there are only finitely many such solutions, and for sufficiently small $\epsilon > 0$,
	\begin{equation}
		HF^{[c-\epsilon,c+\epsilon)}_*(H) = \bigoplus_i HF^{\text{loc}}_*(H,\gamma_i).
	\end{equation}
	In particular, if all one-periodic solutions $\gamma$ of $X_H$ are isolated and $HF^{\text{loc}}_k(H,\gamma) = 0$ for some $k$ and all $\gamma$, then $HF_k(H) = 0$.
\end{lemma}

If $\phi_H^t(z_0)$ is an isolated one-periodic solution of $X_H$, we write $HF^{\text{loc}}_*(H,z_0)$ for $HF^{\text{loc}}_*(H,\phi_H^t(z_0))$. The \emph{support} of $HF^{\text{loc}}_*(H,z_0)$ is the set of integers $k$ for which $HF^{\text{loc}}_k(H,z_0) \neq 0$, denoted by
\begin{equation}\label{eq:suppHF}
	\mathrm{supp}\, HF^{\text{loc}}_*(H,z_0) = \{ k \in \mathbb{Z} : HF^{\text{loc}}_k(H,z_0) \neq 0 \}.
\end{equation}
Since $HF^{\text{loc}}_*(H,z_0)$ is finitely generated, its support is contained in a finite interval:
\begin{equation}\label{eq:suppsub}
	\mathrm{supp}\, HF^{\text{loc}}_*(H,z_0) \subset [\hat{i}_H(z_0) - n, \hat{i}_H(z_0) + n].
\end{equation}
We denote this interval by $\Delta(z_0,H)$.

\subsection{Exponential representation of symplectic matrices}
\label{sec:Exponential representation of symplectic matrices}

Let $\mathcal{L}(\mathbb{R}^{2n})$ denote the group of all $2n \times 2n$ matrices under standard matrix multiplication. A matrix $M \in \mathcal{L}(\mathbb{R}^{2n})$ is symplectic if it satisfies
\[
M^T J_0 M = J_0,
\]
where $J_0 = \begin{pmatrix} 0 & -I_n \\ I_n & 0 \end{pmatrix}$ is the standard symplectic matrix. It should be noted that, for the convenience of this paper,  the dimension of the standard symplectic matrix $J_0$ is not necessarily $2n\times 2n$, which depends on the dimension of the matrix multiplied by it.  The set of all $2n \times 2n$ symplectic matrices forms a subgroup denoted $\mathrm{Sp}(2n)$, called the symplectic group. We denote  $\bR^{-}$ as the closed negative real axis, $\bR^{+}$ as the closed positive real axis, and  $\Un=\{ z\in \bC \,| \,|z|=1 \}$ as the unit circle on the complex plane $\bC$. 

 In \cite{LY2012}, the normal forms of symplectic matrices with distinct eigenvalues are introduced. For the sake of conciseness of this paper,  we briefly enumerate the normal forms of symplectic matrices with distinct eigenvalues; see Appendix \ref{Normal forms of symplectic matrices} for details.
           
\begin{enumerate}
	\item Normal forms for the eigenvalues $\pm 1$: \begin{equation}
		N_1(\pm 1,b)\quad b=\pm 1,0 \quad or \quad  N_m(\pm 1,b) \quad m\ge 2,\,  b=(b_1,\dots,b_m)\in \bR^m.
	\end{equation}
	\item  Normal forms for eigenvalues in $\Un \setminus \bR$ : \begin{equation}
		R(\hat{\theta})\quad \text{or} \quad N_{2m}(\hat{\omega},b)\quad \text{or} \quad N_{2m+1}(\hat{\omega},b) \quad m\ge 1,
	\end{equation}
	where $\hat{\omega}=e^{i\hat{\theta} }$,\, $\hat{\theta}=\theta$ or $-\theta$ and $-\pi< \theta < \pi$.
	\item Nnormal forms for eigenvalues pair $ \lbrace\lambda,\lambda^{-1} \rbrace \subset \bR\setminus \lbrace 0,\pm 1 \rbrace$:
	\begin{equation}
		M_m(\lambda)=\begin{pmatrix} A_m(\lambda) & 0 \\ 0 & C_m(\lambda)	\end{pmatrix}\quad m\ge 1.	
	\end{equation} 
	\item Normal forms for eigenvalue quadruple $\{\rho \omega, \rho \overline{\omega}, \rho^{-1}\omega, \rho^{-1}\overline{\omega}\}\subset   \bC \setminus (\Un \cup \bR)$:
	$$N_{2m}(\rho,\theta)\quad m\ge 1,$$
	where $\rho \in \bR^+\setminus \lbrace 0,  1 \rbrace$, and $\omega=e^{i\theta}\in \Un \setminus \bR $.
\end{enumerate}

We now define the $\diamond$-product operation. For square block matrices
\[
M_1 = \begin{pmatrix} A_1 & B_1 \\ C_1 & D_1 \end{pmatrix}_{2i \times 2i}, \quad
M_2 = \begin{pmatrix} A_2 & B_2 \\ C_2 & D_2 \end{pmatrix}_{2j \times 2j},
\]
their $\diamond$-product is the $2(i+j) \times 2(i+j)$ matrix
\[
M_1 \diamond M_2 = \begin{pmatrix}
	A_1 & 0 & B_1 & 0 \\
	0 & A_2 & 0 & B_2 \\
	C_1 & 0 & D_1 & 0 \\
	0 & C_2 & 0 & D_2
\end{pmatrix}.
\]

\begin{thm}[\cite{LY2012}]\label{symp tran}
	For any $M \in \mathrm{Sp}(2n)$, there exist $P \in \mathrm{Sp}(2n)$, an integer $p \in [0,n]$, and normal forms $M_i \in \mathrm{Sp}(2k_i)$ (with eigenvalues $\lambda_i$ as above) such that $\sum_{i=1}^p k_i = n$ and
	\[
	P^{-1} M P = M_1 \diamond \cdots \diamond M_p.
	\]
\end{thm}

A logarithm of $A \in \mathbb{C}^{n \times n}$ is any matrix $X$ satisfying $e^X = A$. Every nonsingular matrix has infinitely many logarithms.

\begin{thm}[Principal logarithm \cite{HNJ2006}]\label{La:ex}
	Let $A \in \mathbb{C}^{n \times n}$ have no eigenvalues on $\mathbb{R}^-$. Then there exists a unique logarithm $X$ of $A$ whose eigenvalues lie in the strip $\{z \in \mathbb{C} : -\pi < \mathrm{Im}(z) < \pi\}$. This $X$ is called the principal logarithm of $A$, denoted $X = \log(A)$. If $A$ is real, then $\log(A)$ is real. Moreover, if $A$ is symplectic, then $\log(A)$ is infinitesimally symplectic, i.e., satisfies $J_0 X + X^T J_0 = 0$.
\end{thm}

\begin{rmk}
	Throughout this paper, "$\log$" denotes the principal logarithm. For $A \in \mathbb{C}^{n \times n}$ with no eigenvalues on $\mathbb{R}^-$, we have the integral representation:
	\begin{equation}
		\log(A) = \int_0^1 (A - I)[t(A - I) + I]^{-1} \, dt.
	\end{equation}
	For $\alpha \in [-1, 1]$, it holds that $\log(A^\alpha) = \alpha \log(A)$; in particular, $\log(A^{-1}) = -\log(A)$. Explicit expressions for $\log(M)$ for certain normal forms $M$ are provided in Appendix \ref{The precise expressions of certain logarithms}.
	
	Now suppose $A$ is a symplectic matrix, so $A = J_0^{-1}(A^T)^{-1} J_0$. If $A = e^X$, then
	\begin{align*}
		X = \log(A) &= J_0^{-1} \log\left((A^T)^{-1}\right) J_0 = -J_0^{-1} \log(A^T) J_0 = -J_0^{-1} X^T J_0,
	\end{align*}
	which implies $J_0 X + X^T J_0 = 0$.
\end{rmk}

\begin{rmk}
	According to \cite{HNJ2006}, a nonsingular matrix $A \in \mathbb{R}^{n \times n}$ has a real logarithm if and only if for every negative eigenvalue, the number of Jordan blocks of each size is even. 
	
	For example, consider $A = \begin{pmatrix} -1 & -1 \\ 0 & -1 \end{pmatrix}$. The eigenvalue $-1$ has only one Jordan block, so $A$ has no real logarithm. Moreover, the complex logarithms of $A$ are given by
	\[
	X = \begin{pmatrix} i(2k+1)\pi & 1 \\ 0 & i(2k+1)\pi \end{pmatrix}, \quad k \in \mathbb{Z}.
	\]
	However, complex matrices are undesirable in our setting, as we require $J_0 X$ to be real symmetric. Hence, we avoid expressing symplectic matrices with negative real eigenvalues in exponential form.
\end{rmk}
 
 Additionally, when the normal form is $M = N_{2m}(\hat{\omega},b)$ or $N_{2m+1}(\hat{\omega},b)$, indicating that all eigenvalues lie in $\Un \setminus \bR$, certain blocks within $M$ remain undetermined, as specified in Equation \eqref{eq:N2m+1} of Appendix \ref{Normal forms of symplectic matrices}. Consequently, a complete characterization of the properties of $\log(M)$ is not feasible. In the following, we demonstrate that there exists a symplectic matrix $P$ such that $P^{-1}\log(M)P$ admits an explicit representation, thereby improving the structural and analytic properties of $\log(P^{-1}MP)$.
 
 If the normal form is $M = R(\hat{\theta})$, $N_{2m}(\hat{\omega},b)$, or $N_{2m+1}(\hat{\omega},b)$, then the eigenvalues of $\log(M)$ are $\pm i\theta$, where $\theta \neq 0$ and $\theta \neq \pi$. It is known from \cite{BNRC1974} and \cite{CKM1999} that there exists a symplectic transformation $P$ such that $$P^{-1}\log(M)P = m^1 \diamond \dots \diamond m^{s},$$ where each $m^{j}$ $(j=1,\dots,s)$ is a $(2t_j+2)\times(2t_j+2)$ matrix satisfying $J_0 m^j + (m^j)^T J_0 = 0$.
 
 Additionally, each $m^j$ can be decomposed as $m^j=V(\theta)+G$, where $V(\theta)$ is semisimple, $G$ is nilpotent, and the following hold:
 \begin{equation}
 	V(\theta)^T=-V(\theta), \quad V(\theta)G=GV(\theta), \quad J_0V(\theta)+V(\theta)^TJ_0=0, \quad J_0G+G^TJ_0=0.
 \end{equation}
 
  Moreover, the eigenvalues of $V(\theta)+G$ coincide with those of $V(\theta)$, and consequently, $e^{V(\theta)+G}$ and $e^{V(\theta)}$ share the same eigenvalues. Since $V(\theta)$ is semisimple, there exists an invertible complex matrix $P$ such that $P^{-1}V(\theta)P = D$, where 
 \[
 D = \begin{pmatrix}
 	i\theta I_{t_j+1} & 0 \\
 	0 & -i\theta I_{t_j+1}
 \end{pmatrix}
 \]
 is diagonal. Let $M = P^{-1}GP$. The commutation relation $V(\theta)G = GV(\theta)$ implies $DM = MD$, which forces $M$ to be block-diagonal:
 \[
 M = \begin{pmatrix}
 	M_{11} & 0 \\
 	0 & M_{22}
 \end{pmatrix},
 \]
 where each $M_{ii}$ ($i = 1, 2$) is a $(t_j+1) \times (t_j+1)$ matrix. The nilpotency of $G$ implies each $M_{ii}$ is nilpotent. Thus, over $\mathbb{C}$, there exist invertible matrices $Q_i$ such that $Q_i^{-1}M_{ii}Q_i$ is strictly upper triangular. Setting $Q = \begin{pmatrix} Q_1 & 0 \\ 0 & Q_2 \end{pmatrix}$, we have $Q^{-1}MQ$ strictly upper triangular and $Q^{-1}DQ = D$. Define $R = PQ$. Then:
 \begin{equation}
 	R^{-1}V(\theta)R = D\,\, \text{(diagonal)},\quad
 	R^{-1}GR = Q^{-1}MQ \,\, \text{(strictly upper triangular)}.
 \end{equation}
 Hence,
 \begin{equation}\label{eq:eigVG}
 	R^{-1}(V(\theta)+G)R = D + R^{-1}GR
 \end{equation}
 is upper triangular with diagonal entries matching those of $D$, confirming that $V(\theta)+G$ and $V(\theta)$ have identical eigenvalues.
 
 Let $\epsilon^2=1$. When $t_j$ is odd, $m^j$ can be represented as
 \begin{equation}\label{odd mt}
 	\left(
 	\begin{array}{cccc|cccc}
 		L & & & & 0 & & & \\
 		I & L & & & & 0 & & \\
 		& \ddots & \ddots & & & & \ddots & \\
 		& & I & L & & & & 0\\ 
 		\hline
 		0 & & & & L & -I & &  \\	
 		& \ddots & & & & \ddots & \ddots    \\
 		& & 0 & & & & L & -I    \\
 		& & & \Delta & & & & L      \\ 
 	\end{array}
 	\right),
 \end{equation}
 where $L=\begin{pmatrix}
 	0 & -\theta \\
 	\theta & 0 \\ 
 \end{pmatrix}$, $I=\begin{pmatrix}
 	1 & 0 \\
 	0 & 1\\
 \end{pmatrix}$, and $\Delta=\begin{pmatrix}
 	(-1)^{r-1}\epsilon & 0 \\
 	0 & (-1)^{r-1}\epsilon \\
 \end{pmatrix}$ with $r=\frac{t_j+1}{2}$. 
 
 When $t_j$ is even, $m^j$ can be represented as   
 \begin{equation}\label{even mt}
 	\left(
 	\begin{array}{ccccc|ccccc}
 		0 & & & & & & & & & -\epsilon \theta \\	
 		1 & 0 & & & & & & & \epsilon \theta & \\	
 		&  \ddots& \ddots & & & & &  \begin{rotate}{70}$\ddots$\end{rotate} &   \\	
 		& & 1 & 0 & &  &  \epsilon \theta  & & & \\	
 		& & & 1 & 0  &  -\epsilon \theta &  & & & \\
 		\hline
 		& & &  & \epsilon \theta  &  0 & -1 & & & \\
 		& &  & -\epsilon \theta  & & &  0 & -1  & &  \\
 		& & \begin{rotate}{70}$\ddots$\end{rotate}  & & & & &  \ddots & \ddots  &  \\
 		& -\epsilon \theta  & & & & & & & 0 & -1   \\
 		\epsilon \theta  & & & & & & & & & 0    \\	
 	\end{array}
 	\right)	.
 \end{equation}
The semisimple matrix $V(\theta)$ can be expressed respectively as follows:
 \begin{equation}
 	\left(
 	\begin{array}{cccc|cccc}
 		L & & & & 0 & & & \\
 		& L & & & & 0 & & \\
 		&  & \ddots & & & & \ddots & \\
 		& &  & L & & & & 0\\ 
 		\hline
 		0 & & & & L &  & &  \\	
 		& \ddots & & & & \ddots &    \\
 		& & 0 & & & & L &     \\
 		& & & 0 & & & & L      \\ 
 	\end{array}
 	\right),
 	\quad
 	\left(
 	\begin{array}{cccc|cccc}
 		0 & & & & & & & -\epsilon \theta \\	
 		& 0 & & & & & \epsilon \theta & \\	
 		&  & \ddots & & &  \begin{rotate}{70}$\ddots$\end{rotate} &   \\	
 		& & & 0  &  -\epsilon \theta &  & & \\
 		\hline
 		& & & \epsilon \theta  &  0 &  & & \\
 		& &  -\epsilon \theta  & & &  0 &  &  \\
 		& \begin{rotate}{70}$\ddots$\end{rotate} & & & & &  \ddots &  \\
 		\epsilon \theta  & & & & & & & 0    \\	
 	\end{array}
 	\right).
 \end{equation}
 
 Since $e^{V(\theta)+G}=e^{V(\theta+2k\pi)+G}$ for any $k\in \bZ$, we may assume without loss of generality that $\theta \in (-\pi,0)\cup (0,\pi)$, analogous to Theorem \ref{La:ex}. Furthermore, it can be verified that $e^{V(\theta)}$ is a unitary matrix, which can respectively be expressed as:
 \begin{equation}
 	e^{V(\theta)}=\text{diag}\left(R(\theta), \cdots, R(\theta)\right) \quad \text{with} \quad R(\theta)=\begin{pmatrix}
 		\cos(\theta) & -\sin(\theta) \\
 		\sin(\theta) & \cos(\theta)
 	\end{pmatrix}.
 \end{equation}
or 
\begin{equation}
	e^{V(\theta)}=\cos(\theta)I_{2(t_j+1)}+\frac{\sin(\theta)}{\epsilon \theta}V(\theta).
\end{equation}

 By Theorem \ref{La:ex}, if $M_i$ is a normal form with eigenvalues $\lambda_i \notin (\mathrm{U} \setminus \{\pm 1\}) \cup \mathbb{R}^-$, then there exists $m_i = \log(M_i)$ with $M_i = e^{m_i}$. However, for eigenvalues $\lambda_i \in \mathbb{R}^-$, the logarithm of the normal form $M_i$ may not be real. In this case, $m_i = \log(-M_i)$ exists and $M_i = -e^{m_i}$. Furthermore, each $m_i$ is infinitesimally symplectic, satisfying $J_0 m_i + m_i^T J_0 = 0$. These results are summarized in the following theorem.

\begin{thm}
 For any $M \in \mathrm{Sp}(2n)$, there exist $P \in \mathrm{Sp}(2n)$ and integers $p,q,s \in [0,n]$ such that
\begin{equation}\label{expbie}
	P^{-1}MP = (-e^{m_1}) \diamond \cdots \diamond (-e^{m_p}) \diamond e^{\hat{m}_1} \diamond \cdots \diamond e^{\hat{m}_q} \diamond e^{m^1} \diamond \cdots \diamond e^{m^s},
\end{equation}
where:
\begin{itemize}
	\item $-e^{m_i}$ ($i=1,\dots,p$) are normal forms with eigenvalues $\lambda_i \in \mathbb{R}^-$,
	\item $e^{\hat{m}_l}$ ($l=1,\dots,q$) are normal forms with eigenvalues $\lambda_l \notin (\mathrm{U} \setminus \{\pm 1\}) \cup \mathbb{R}^-$,
	\item $m^j$ ($j=1,\dots,s$) are defined by \eqref{odd mt} or \eqref{even mt} , and $e^{m^j}$ are normal forms with purely imaginary eigenvalues.
\end{itemize}
Furthermore, the eigenvalues of $m_i$, $\hat{m}_l$, and $m^j$ all lie in the strip $\{z \in \mathbb{C} : -\pi < \mathrm{Im}(z) < \pi\}$, and each matrix is infinitesimally symplectic (satisfying $J_0 m + m^T J_0 = 0$).
	 
\end{thm} 
 
Based on the above analysis, a symplectic transformation can be performed, under which the periodic solutions before and after the transformation are in one-to-one correspondence. Without loss of generality, assume that $\phi^1_Q$ admits the expression on the right-hand side of equation \eqref{expbie}.

\section{Construction and Properties of the Functions}
\label{Construction and Properties of the Functions}

As we desire to obtain an infinite number of periodic solutions, it is necessary for us to consider the iteration of the Hamiltonian $H$. Define $H^{\times k}$ as
\begin{equation}
	H^{\times k}(t,z)=kH(kt,z),\quad k\in \bN^+,		
\end{equation} 
where $\bN^+$ represents the set of all natural numbers that are strictly greater than zero.
Then $\phi^t_{H^{\times k}}(z_0) = \phi^{kt}_H(z_0)$. If $\overline{z}(t)$ is a one-periodic solution of $X_H$, then $\overline{z}(kt)$ is a one-periodic solution of $X_{H^{\times k}}$ with:
\[
\mathcal{A}_{H^{\times k}}(\overline{z}(kt)) = k\mathcal{A}_H(\overline{z}(t)), \quad 
i_{H^{\times k}}(\overline{z}(kt)) = i_H(\overline{z}(t),k), \quad 
i_{\infty}(H^{\times k}) = i_{\infty}(H,k).
\]

For one-periodic Hamiltonians $F_t(z)$ and $G_t(z)$, define the following operations \cite{ML2024}:
\begin{equation}
	\overline{F}_t(z)=-F_t\left(\phi^t_F(z)\right).
\end{equation}
\begin{equation}
	(F\#G)_t(z)=F_t(z)+G_t\left((\phi^t_F)^{-1}(z)\right).
\end{equation}
\begin{equation}
	(F\wedge G)_t(z)=\begin{cases}
		2\rho'(2t)G_{\rho(2t)}(z), & t\in[k,k+\frac{1}{2}] \\
		2\rho'(2t-1)F_{\rho(2t-1)}(z), & t\in[k+\frac{1}{2},k+1]
	\end{cases}, \, k\in \bZ.
\end{equation}
where $\rho \in C^{\infty}(\mathbb{R}, [0,1])$ is a 2-periodic function, non-decreasing on $[0,1]$, symmetric ($\rho(t) = \rho(2-t)$ for $t \in [0,1]$), with $\rho(0) = 0$, $\rho(1) = 1$, and $\rho' < 2$.

The corresponding Hamiltonian flows are:
\begin{equation}
	\phi^t_{\overline{F}} = (\phi^t_F)^{-1}, \quad
	\phi^t_{F \# G} = \phi^t_F \circ \phi^t_G, \quad
	\phi^t_{F \wedge G} = 
	\begin{cases}
		\phi^{\rho(2t)}_G, & t \in [k, k+\tfrac{1}{2}], \\
		\phi^{\rho(2t-1)}_F \circ \phi^1_G, & t \in [k+\tfrac{1}{2}, k+1].
	\end{cases}
\end{equation}

Note that $F \wedge G$ is one-periodic in $t$, while $\overline{F}$ and $F \# G$ generally are not. However, when $F_t$ is a quadratic form with $\phi^1_F = I_{2n}$, both become one-periodic.

Now let $P^{\mu}_t(z) = \frac{1}{2}\langle B^{\mu}_t z, z \rangle$ be a quadratic form generating a loop of Maslov index $\mu$, where $B^{\mu}_t$ is a real symmetric matrix with $B^\mu_{t+1} = B^\mu_t$, and $\phi^t_{P^{\mu}} \in \mathrm{Sp}(2n)$ satisfies $\phi^0_{P^{\mu}} = \phi^1_{P^{\mu}} = I_{2n}$ with Maslov index $\mu$. Then:
\[
\overline{P^\mu_t}(z) = -\frac{1}{2}\langle B^{\mu}_t \phi^t_{P^{\mu}} z, \phi^t_{P^{\mu}} z \rangle
\]
is also a quadratic generating loop, with fundamental solution $(\phi^{t}_{P^{\mu}})^{-1}$ and Maslov index $-\mu$.
The Maslov index here follows \cite{GJ2012} but with our sign convention: we take the negative of the original definition, so it counts clockwise rotations of certain eigenvalues.

\begin{lemma}\label{le:Q}
		Let $H:S^1\times \bR^{2n}\rightarrow \bR$ be a smooth Hamiltonian that equal to a  quadratic form $Q_t(z)$ at infinity, and $\phi^1_{Q}$ can be expressed as $e^{\hat{m}}$ or $-e^{\hat{m}}$. Then there exists a quadratic form $P_t(z)$ generating a loop  such that the quadratic form of $P_t \# H_t = \hat{Q}_t + \hat{h}_t$ admits an explicit expression. More precisely, when $\phi^1_{Q} = e^{\hat{m}}$,  
	\begin{equation}\label{defi Q1}
		\hat{Q}(z) = \frac{1}{2}\left\langle J_0\hat{m}z, z\right\rangle,
	\end{equation} 
	 which is time-independent. When $\phi^1_{Q} = -e^{\hat{m}}$, 
	\begin{equation}\label{defi Q2}
		\hat{Q}(z) = \frac{1}{2}\left\langle -\pi I_{2n}z, z\right\rangle + \frac{1}{2}\left\langle J_0\hat{m}e^{-\pi J_0 t}z, e^{-\pi J_0 t}z\right\rangle,    
	\end{equation}
	which is time-dependent.
\end{lemma}
\begin{proof}
			
Case 1: 	
Let $\phi^1_{Q} = e^{\hat{m}}$, where $e^{\hat{m}}$ is a symplectic matrix satisfying $J_0\hat{m} + \hat{m}^T J_0 = 0$. Define $B = J_0\hat{m}$, which is symmetric. Consider the Hamiltonians:	
	$$
Q_1(z) = \frac{1}{2}\left\langle Bz, z\right\rangle, \quad P_t(z) = Q_1 \# \overline{Q_t}.
$$	
Then $\phi^1_{Q_1} = e^{\hat{m}}$, and we obtain:	
$$P_t(z) = Q_1(z) - Q_t\left(\phi^t_{Q} \circ \phi^{-t}_{Q_1}(z)\right).$$	
Since both $Q_1$ and $Q_t$ are one-periodic in $t$, their flows satisfy the following decomposition for any $j \in \mathbb{N}^+$ and $j \le t \le j+1$:

\begin{equation}
\phi^t_{Q_1} = \phi^{(t-j)}_{Q_1} (\phi^1_{Q_1})^j, \quad \phi^t_{Q} = \phi^{(t-j)}_{Q} (\phi^1_{Q})^j .
\end{equation}
For $t+1 \in [j+1, j+2]$, we have:	 
\begin{equation}
\phi^{t+1}_{Q_1} = \phi^{(t-j)}_{Q_1} (\phi^1_{Q_1})^{j+1}=\phi^t_{Q_1}\phi^1_{Q_1}, \, \quad \phi^{t+1}_{Q} = \phi^{(t-j)}_{Q} (\phi^1_{Q})^{j+1}=\phi^t_{Q}\phi^1_{Q}.	
\end{equation}
It follows that:	
	$$
\phi^{t+1}_{Q} \circ \phi^{-(t+1)}_{Q_1} =\phi^t_{Q}\phi^1_{Q}\circ \phi^{-1}_{Q_1}\phi^{-t}_{Q_1}
= \phi^{t}_{Q} \circ \phi^{-t}_{Q_1}.
$$	
Hence, $P_t(z)$ is one-periodic in $t$. Moreover, the time-one flow of $P$ satisfies:	
$$\phi^1_P = \phi^1_{Q_1} \circ (\phi^1_{Q})^{-1} = I.$$	
Then we have
	\begin{equation}
	\begin{aligned}
		P \# H_t &= P_t + H_t \circ (\phi^t_P)^{-1} = Q_1 - Q_t \circ \phi^t_{Q} \circ \phi^{-t}_{Q_1} + Q_t \circ \phi^t_{Q} \circ \phi^{-t}_{Q_1} + h_t \circ (\phi^t_P)^{-1} \\
		&= Q_1 + h_t \circ (\phi^t_P)^{-1} = \frac{1}{2}\left\langle Bz, z\right\rangle + h_t \circ (\phi^t_P)^{-1}.
	\end{aligned}
\end{equation}

Case 2: 
Let $\phi^1_{Q} = -e^{\hat{m}}$, set $B = J_0\hat{m}$, then $B$ is a symmetric matrix. 
Define  
$$
Q_1(z) = \frac{1}{2}\left\langle -\pi I_{2n}z, z\right\rangle, \,\, Q_2(z) = \frac{1}{2}\left\langle Bz, z\right\rangle, \,\, Q'_t(z) = Q_1 \# Q_2, \,\, P_t(z) = Q'_t \# \overline{Q}_t.
$$	
Then
\begin{equation}
	Q'_t(z) = \frac{1}{2}\left\langle -\pi I_{2n}z, z\right\rangle + \frac{1}{2}\left\langle J_0\hat{m}e^{-\pi J_0 t}z, e^{-\pi J_0 t}z\right\rangle,
\end{equation}
which is one-periodic in $t$. Similarly, $P_t(z)$ is one-periodic in $t$. The time-one flows satisfy:
\begin{equation}
	\phi^1_{Q'} = \phi^1_{Q_1} \circ \phi^1_{Q_2} = -e^{\hat{m}}, \quad \phi^1_P = \phi^1_{Q'} \circ (\phi^1_{Q})^{-1} = I.
\end{equation}
 Therefore, we have
\begin{equation}
	\begin{aligned}
		P \# H_t &= P_t + H_t \circ (\phi^t_P)^{-1} = Q'_t - Q_t \circ \phi^t_{Q} \circ \phi^{-t}_{Q'} + Q_t \circ \phi^t_{Q} \circ \phi^{-t}_{Q'} + h_t \circ (\phi^t_P)^{-1} \\
		&= Q'_t + h_t \circ (\phi^t_P)^{-1} = \frac{1}{2}\left\langle -\pi I_{2n}z, z\right\rangle + \frac{1}{2}\left\langle J_0\hat{m}e^{-\pi J_0 t}z, e^{-\pi J_0 t}z\right\rangle + h_t \circ (\phi^t_P)^{-1}.
	\end{aligned}
\end{equation}
\end{proof}

Suppose that $\phi^1_{Q} = M_1 \diamond M_2$, where $M_1 \in \Sp(2i)$ and $M_2 \in \Sp(2j)$, and there exist matrices $\hat{m}_1$ and $\hat{m}_2$ such that $M_1 = e^{\hat{m}_1}$ and $M_2 = -e^{\hat{m}_2}$. Then we have the decomposition:
$$M_1 \diamond M_2 = (I_{2i} \diamond (-I_{2j})) e^{\hat{m}_1 \diamond \hat{m}_2}.$$
Let us denote the coordinates by $z = (z_{11}, z_{21}, z_{12}, z_{22})$, and define $z_1 = (z_{11}, z_{12})$ and $z_2 = (z_{21}, z_{22})$. Then the quadratic form $P_t(z)$ associated with the generating loop admits a decomposition:
 $$P_t(z) = P^1_t(z_1) + P^2_t(z_2),$$
where for $i = 1, 2$, the component $P^i_t(z_i)$ is defined in terms of $M_i$ according to Lemma~\ref{le:Q}.
Furthermore, the quadratic form of $P_t \# H_t$ at infinity takes the form:
\begin{equation}\label{dandu}
	\hat{Q}_t(z) = \hat{Q}^1_t(z_1) + \hat{Q}^2_t(z_2),
\end{equation}
where for $i = 1, 2$, the term $\hat{Q}^i_t(z_i)$ denotes the quadratic form associated with $M_i$ as defined in Lemma~\ref{le:Q}.
Since we assume that $\phi^1_Q$ can be expressed as the right-hand side of equation~\eqref{expbie}, it follows that there exists a quadratic form $P_t$ associated with a generating loop such that the quadratic form of $P_t \# H_t$ admits an explicit representation as the sum of the expressions given in~\eqref{defi Q1} and~\eqref{defi Q2}.

\begin{lemma}\label{le:PandH}
	Let $H_t = Q_t + h_t$ be a Hamiltonian that equal to a quadratic form  $Q_t$ at infinity, and let $P^{\mu}_t(z) = \frac{1}{2} \langle B^{\mu}_t z, z \rangle$ be a quadratic generating loop of Maslov index $\mu$. Then:
	\begin{enumerate}
		\item The Conley--Zehnder indices at infinity satisfy:
		\begin{align*}
			i_{\infty}(P^\mu \# H) &= i_{\infty}(H) + 2\mu, \\
			i_{\infty}(P^\mu \# H, s) &= i_{\infty}(H, s) + 2\mu s \quad (s \in \mathbb{N}^+), \\
			\hat{i}_{\infty}(P^\mu \# H) &= \hat{i}_{\infty}(H) + 2\mu.
		\end{align*}
		\item The time-$1$ maps coincide: $\phi^1_{P^\mu \# H} = \phi^1_H$. Moreover, for every $z_0 \in \operatorname{Fix}(\phi^1_H)$, we have:
		\begin{equation}\label{eq:ipH}
		i_{P\#H}(z_0)=i_{H}(z_0)+2\mu,
	\end{equation}
	\begin{equation}\label{eq:ipHs}
		i_{P\#H}(z_0,s)=i_{H}(z_0,s)+2\mu s, \, s\in \bN^+,
	\end{equation}
	\begin{equation}\label{eq:mipH}
		\hat{i}_{P\#H}(z_0)=\hat{i}_{H}(z_0)+2\mu,
	\end{equation}		
	\begin{equation}\label{eq:ApH}
		\cA_{P\#H}(z_0)=\cA_{H}(z_0).
	\end{equation}
		\item Define $H^{k \ominus l} = (\overline{P^\mu} \# H^{\times (k-l)}) \wedge H^{\times l}$ for $k > l$ in $\mathbb{N}^+$. Then:
		\[
		\phi^1_{H^{k \ominus l}} = \phi^1_{H^{\times k}}.
		\]
		If $z_0 \in \operatorname{Fix}(\phi^1_H)$, then $z_0 \in \operatorname{Fix}(\phi^1_{H^{k \ominus l}}) = \operatorname{Fix}(\phi^1_{H^{\times k}})$, and:
	\begin{equation}
		i_{\infty}(H^{k\ominus l})=i_{\infty}(H,k)-2\mu, 		
	\end{equation}	
	\begin{equation}
		i_{H^{k\ominus l}}(z_0)=i_{H^{\times k}}(z_0)-2\mu=i_{H}(z_0,k)-2\mu.
	\end{equation}		
	\begin{equation}
		\cA_{H^{k\ominus l}}(z_0)=k\cA_{H}(z_0).	
	\end{equation}
	\end{enumerate}
\end{lemma}

\begin{proof}
By definition, we have $$P^\mu\#H = P^\mu_t(z) + Q_t(\gamma^{-t}z) + h_t(\gamma^{-t}z),$$ where $\gamma(t)= \phi^t_{P^\mu} \in \Sp(2n)$ satisfies $\gamma(0) = \gamma(1) = I_{2n}$. The term $h_t(\gamma^{-t}z)$  represents the compactly supported part, while $P^\mu_t(z) + Q_t(\gamma^{-t}z)$ constitutes a quadratic form.
The flow $\gamma(t)\phi^t_Q \in \Sp(2n)$, defined for $t \in [0,s]$, is homotopic to the symplectic path
 \begin{equation}
	\begin{cases}
		\phi_Q^{2t}, & t\in [0,\frac{s}{2}] \\
		\gamma(2(t-\frac{s}{2}))\phi_Q^s, & t\in [\frac{s}{2},s]
	\end{cases},\, s\in \bN^+
\end{equation}
Since $P^\mu_t$ is one-periodic with respect to $t$,  for $s\le t\le (s+1)$, 
\begin{equation}
	\gamma(t)=\gamma(t-s).
\end{equation} 
According to the properties of the Conley-Zehnder index, we have 
\begin{equation}
	i_{\infty}(P^\mu\#H,s)	=i_{\infty}(H,s)+2\mu s,
\end{equation}
and consequently,  $$\hat{i}_{\infty}(P^\mu\#H)=\hat{i}_{\infty}(H)+2\mu.$$
		
Since $\phi^t_{P^\mu\#H} = \gamma(t) \circ \phi^t_H$ and $\gamma(1) = I_{2n}$, it follows that $\phi^1_{P^\mu\#H} = \phi^1_H$. For every $z_0 \in  \Fix \phi^1_H$, let $\overline{z} = \phi^t_H(z_0)$.  Linearizing along $\overline{z}$ gives $\widetilde{\gamma}(t)$ satisfying
 \begin{equation}
 	\dot{\widetilde{\gamma}}(t) = -J_0 \frac{\partial^2}{\partial z^2} H_t(\overline{z}) \widetilde{\gamma}(t).
 \end{equation}
Linearizing along $\gamma(t)\overline{z}$ for $P^\mu\# H$ yields 
\begin{equation}\label{pzlinear}
	\dot{z}(t) = -J_0\left(B^\mu_t + (\gamma^{-t})^T \frac{\partial^2}{\partial z^2} H_t(\overline{z}) \gamma^{-t}\right) z(t).
\end{equation}
A direct computation shows $\gamma \widetilde{\gamma}$ is the fundamental solution, hence by the Conley-Zehnder index property, the equations  \eqref{eq:ipH}, \eqref{eq:ipHs} and \eqref{eq:mipH} can be proved.

 Moreover, $\gamma(t)$ is a path of symplectic matrices satisfying $\gamma(t)^T J_0 \gamma(t) = J_0$. Therefore, 
\begin{equation}
\begin{aligned}
\cA_{P^\mu\#H}\left(\gamma(t)\overline{z}\right)=&  \frac{1}{2}\int_{S^1}\left\langle B^\mu_t\gamma(t)\overline{z}, \gamma(t)\overline{z}\right\rangle dt+
   \frac{1}{2}\int_{S^1}\left\langle J_0\dot{\overline{z}},\overline{z}\right\rangle dt \\
& -\int_{S^1}\frac{1}{2} \left\langle B^{\mu}_t\gamma(t)\overline{z},\gamma(t)\overline{z}\right\rangle + H_t(\overline{z})dt \\
&= \frac{1}{2}\int_{S^1}\left\langle J_0\dot{\overline{z}}, \overline{z}\right\rangle dt-H_t(\overline{z})dt=\cA_{H}\left(\overline{z}\right).
\end{aligned}		
\end{equation}

Now define 
\begin{equation}
	\overline{P^\mu}\# H^{\times (k-l)}=-\frac{1}{2}\left\langle B^\mu_t\phi^t_{P^\mu}z,\phi^t_{P^\mu}z\right\rangle +Q^{\times(k-l)}\left(\phi^t_{P^\mu}z\right)+h^{\times(k-l)}(\phi^t_{P^\mu}z)	.
\end{equation}
The quadratic form at infinity is $Q^{k\ominus l}=(\overline{P^\mu}\# Q^{\times (k-l)})\wedge Q^{\times l}$, with  flow homotopic to $\left( \phi^t_{P^\mu}\right)^{-1}\phi^t_{Q^{\times (k-l)}}\phi^t_{Q^{\times l}}$, so 
\begin{equation}
	i_{\infty}(H^{k\ominus l}) = i_{\infty}(H,k) - 2\mu.
\end{equation}
The flow is explicitly
\begin{equation}
	\phi^t_{H^{k\ominus l}}=\begin{cases}
		\phi^{\rho(2t)}_{H^{\times l}}, & t\in [0,\frac{1}{2}] \\
		(\phi^{\rho(2t-1)}_{P^\mu})^{-1}\circ \phi^{\rho(2t-1)}_{H^{\times (k-l)}}\circ \phi^1_{H^{\times l}}, & t\in [\frac{1}{2},1]
	\end{cases},
\end{equation}
giving  $\phi^1_{H^{k\ominus l}}=\phi^1_{H^{\times k}}$. By homotopy, for $z_0 \in \Fix \phi^1_H$,
\begin{equation}
	i_{H^{k\ominus l}}(z_0)=i_{H}(z_0,k)-2\mu.
\end{equation}
The action computation simplifies via variable substitutions:
\begin{equation}
\cA_{H^{k\ominus l}}(z_0)=\cA_{H^{\times l}}(z_0)+\cA_{\overline{P^\mu}\# H^{\times (k-l)}}(z_0) 
=l\cA_{H}(z_0)+(k-l)\cA_{H}(z_0)=k\cA_{H}(z_0). 	
\end{equation}

\end{proof}

From Lemma~\ref{le:PandH}, we conclude 
 $$
 i_{\infty}(P\#H) - i_{\infty}(H) = i_{P\#H}(z_0) - i_H(z_0) = \hat{i}_{P\#H}(z_0) - \hat{i}_H(z_0),
 $$
 and $\cA_{P\#H}(z_0) = \cA_H(z_0)$ for all $z_0 \in  \Fix\phi^1_H$. Hence, if there exists $z_0 \in \Fix\phi^1_H$ such that $\hat{i}_H(z_0) \ne \hat{i}_{\infty}(H)$, then $\hat{i}_{P\#H}(z_0) \ne \hat{i}_{\infty}(P\#H)$.  
 Furthermore, based on the properties of local homology, we obtain  
 \begin{equation}
 	HF^{loc}_{*+2\mu}\left(P\#H, \phi^t_{P\#H}(z_0)\right) = HF^{loc}_*\left(H, \phi^t_H(z_0)\right),
 \end{equation}
 where $\mu$ is the Maslov index associated with the quadratic form $P$. For more details, see references \cite{GVL2007} and \cite{GVBZ2010}.

If $\phi^1_H$ has an isolated, homologically nontrivial, twist fixed point $z_0$ and finite fixed point set, then so does $\phi^1_{P\#H}$. 
 Therefore, without loss of generality, we may assume that the Hamiltonian function considered in this paper is $P\#H$, whose quadratic form can be expressed as the sum of the expressions in \eqref{defi Q1} and \eqref{defi Q2}. For simplicity, we continue to denote the Hamiltonian $P\#H$ by $H$.
 
Fix two odd numbers $k>l\ge 1$, through the index iteration formula, it can be obtained that the parity of the indices at infinity of the iterates $H^{\times k}$ and $H^{\times l}$ are the same, so that
\begin{equation}
	2\mu =i_\infty(H^{\times k})-i_\infty(H^{\times l}),
\end{equation}
for some $\mu \in \bZ$. Furthermore, we have
\begin{equation}
	(k-l)\hat{i}_{\infty}(H)-n\le 2\mu \le (k-l)\hat{i}_{\infty}(H)+n	.
\end{equation}

\begin{lemma}\label{Qklit}
	For large primes  $k > l$,  let $2\mu = i_\infty(H^{\times k}) - i_\infty(H^{\times l})$. Then there exists a quadratic form $P^\mu_t$  generating a loop of Maslov index $\mu$, such that
	\begin{enumerate}
		\item $\phi^t_{P^\mu}$ is a unitary loop.
		\item $\overline{P^{\mu}} \# Q^{\times (k-l)}$ is time-independent.
		\item $\phi^t_{\overline{P^{\mu}} \# Q^{\times (k-l)}} \phi^1_{Q^{\times l}}\in \Sp(2n)$ is non-degenerate for all $t \in [0,1]$ , meaning that $1$ is not an eigenvalue of this matrix.
	\end{enumerate}
\end{lemma}

\begin{proof}
Assume that $\phi^1_Q$ can be expressed as in equation \eqref{expbie}, and similarly to \eqref{dandu}, it suffices to consider the cases $\phi^1_Q = e^{\hat{m}}$ or $\phi^1_Q = -e^{\hat{m}}$. By definition,
\begin{equation}
	\overline{P^{\mu}} \# H^{\times (k-l)} = \overline{P^{\mu}} \# Q^{\times (k-l)} + h^{\times (k-l)}(\phi^t_{P^\mu} z).
\end{equation}

Case 1: Suppose $\phi^1_Q = M_m(\lambda)$ with $\lambda \in \mathbb{R}^+ \setminus \{0,1\}$, or $N_{2m}(\rho,\theta)$ with $\rho \in \mathbb{R}^+ \setminus \{0,1\}$. Then there exists $\hat{m}$ such that $\phi^1_Q = e^{\hat{m}}$, where the eigenvalues of $\hat{m}$ are $\{\pm \log(\lambda)\}$ or $\{\log(\rho) \pm i\theta, -\log(\rho) \pm i\theta\}$, respectively. In this case, $i_\infty(H^{\times l}) = 0$ for any odd integer $l$. Thus, $P^\mu = 0$ and $\phi^t_{P^\mu} = I_{2m}$, so $\phi^0_{P^\mu} = \phi^1_{P^\mu} = I_{2m}$, and $\phi^t_{P^\mu}$ is a loop of unitary matrices. Moreover,
\begin{equation}
	\overline{P^{\mu}} \# H^{\times (k-l)} = H^{\times (k-l)} = \frac{1}{2}(k-l)\langle J_0\hat{m}z, z\rangle + h^{\times (k-l)}(z).
\end{equation}
Hence, the quadratic form $\overline{P^{\mu}} \# Q^{\times (k-l)}$ is time-independent. Furthermore, $\phi^t_{\overline{P^{\mu}} \# Q^{\times (k-l)}} = e^{(k-l)\hat{m}t}$, and the composed map $\phi^t_{\overline{P^{\mu}} \# Q^{\times (k-l)}} \phi^1_{Q^{\times l}} = e^{((k-l)t + l)\hat{m}}$ is non-degenerate for all $t \in [0,1]$, since the eigenvalues of $((k-l)t + l)\hat{m}$ avoid $2\pi i \mathbb{Z}$ due to the properties of $\hat{m}$ and the choices of $k$ and $l$.	
	
Case 2: Suppose $\phi^1_Q = N_m(-1,b)$ or $M_m(-\lambda)$ with $\lambda \in \mathbb{R}^+ \setminus \{0,1\}$. Then there exists $\hat{m}$ such that $\phi^1_Q = -e^{\hat{m}}$, where the eigenvalues of $\hat{m}$ are $0$  or $\{\pm \log(\lambda)\}$. In this case, $i_\infty(H) = -m$, and for any odd integer $l$, $i_\infty(H^{\times l}) = l \cdot i_\infty(H)$. Thus, $2\mu = -(k-l)m$. Define
\begin{equation}
	P^{\mu} = \frac{1}{2} \left\langle \frac{2\pi \mu}{m} I_{2m} z, z \right\rangle = \frac{1}{2} (k-l) \left\langle -\pi I_{2m} z, z \right\rangle,
\end{equation}
so  $\phi^t_{P^\mu} = e^{-\frac{2\pi \mu}{m} t J_0} = e^{(k-l)\pi J_0 t}$. It satisfies $\phi^0_{P^\mu} = \phi^1_{P^\mu} = I_{2m}$, and $\phi^t_{P^\mu}$ is a path of unitary matrices. Then,
\begin{equation}
	\begin{aligned}
	\overline{P^{\mu}} \# H^{\times (k-l)}= &-\frac{1}{2}(k-l)\left\langle -\pi I_{2m}\phi^t_{P^\mu}z,\phi^t_{P^\mu}z\right\rangle +\frac{1}{2}(k-l)\left\langle -\pi I_{2m}\phi^t_{P^\mu}z,\phi^t_{P^\mu}z\right\rangle  \\ & +\frac{1}{2}(k-l)\left\langle J_0\hat{m} e^{-J_0\pi (k-l)t}\phi^t_{P^\mu}z,e^{-J_0\pi (k-l)t}\phi^t_{P^\mu}\right\rangle +h^{\times (k-l)}(\phi^t_{P^\mu}z) \\
	& =\frac{1}{2}(k-l)\left\langle J_0\hat{m}z,z\right\rangle +h^{\times (k-l)}(\phi^t_{P^\mu}z).
\end{aligned}
\end{equation}
Thus, the quadratic form $\overline{P^{\mu}} \# Q^{\times (k-l)}$ is time-independent. Moreover, we have  $\phi^t_{\overline{P^{\mu}} \# Q^{\times (k-l)}} = e^{(k-l)\hat{m}t}$, and the composition $\phi^t_{\overline{P^{\mu}} \# Q^{\times (k-l)}} \phi^1_{Q^{\times l}} = - e^{((k-l)t + l)\hat{m}}$ is non-degenerate for all $t \in [0,1]$. This holds because the eigenvalues of $- e^{((k-l)t + l)\hat{m}}$ are either $-1$  or negative real numbers .
	
Case 3: Suppose $\phi^1_Q = e^{m^j}$, where$m^j = V(\theta) + G$ with $-\pi < \theta < \pi$, $\theta \ne 0$, is a $2(t_j+1) \times 2(t_j+1)$ matrix defined by \eqref{odd mt} and $t_j$ is odd. Since $V(\theta)$ and $G$ commute, we have $e^{m^j t} = e^{G t} e^{V(\theta) t}$. The path $e^{G t} e^{V(\theta) t}$ for $t \in [0,1]$ is homotopic to
\[
\gamma(t) = 
\begin{cases}
	e^{V(\theta)2t}, & t \in [0, \tfrac{1}{2}], \\
	e^{2(t-\tfrac{1}{2})G} e^{V(\theta)}, & t \in [\tfrac{1}{2}, 1].
\end{cases}
\]
 For $t \in [\tfrac{1}{2},1]$, according to \eqref{eq:eigVG}, the eigenvalues of $\gamma(t)$ are $e^{i\theta}$ and $e^{-i\theta}$. Hence, the Conley–Zehnder index satisfies $i_\infty(H) = -\mathrm{sgn}(\theta)(t_j+1)$, where $"\text{sgn}"$ is the sign function.
 	
For fixed $\theta$, there exists a prime $p$ such that for all primes $l > p$, we have $l\theta \not\equiv 0 \pmod{2\pi}$, ensuring $\phi^l_Q$ is non-degenerate. Moreover,
\[
i_\infty(H,l) = -\mathrm{sgn}(\theta)(t_j+1)\left(2\left\lfloor \frac{l|\theta|}{2\pi} \right\rfloor + 1\right),
\]
where $\lfloor a \rfloor=\max\{k\in \bZ \mid k\le a\}$ for any $a\in \bR$. Therefore
\[
2\mu = i_\infty(H,k) - i_\infty(H,l) = -\mathrm{sgn}(\theta)(t_j+1)\left(2\left\lfloor \frac{k|\theta|}{2\pi} \right\rfloor - 2\left\lfloor \frac{l|\theta|}{2\pi} \right\rfloor\right).
\]

Define
\begin{equation}\label{Pmu}
	P^\mu = \frac{1}{2} \left\langle J_0 V(\theta_\mu) z, z \right\rangle,
\end{equation}
where
\[
\theta_\mu = -\frac{2\pi \mu}{t_j+1} = \mathrm{sgn}(\theta)\pi \left(2\left\lfloor \frac{k|\theta|}{2\pi} \right\rfloor - 2\left\lfloor \frac{l|\theta|}{2\pi} \right\rfloor\right).
\]
Then $\phi^t_{P^\mu} = e^{V(\theta_\mu)t}$ satisfies $\phi^0_{P^\mu} = \phi^1_{P^\mu} = I_{2(t_j+1)}$, and is a continuous path of unitary matrices with Maslov index $\mu$.

Using the properties $V(\theta_\mu)^T = -V(\theta_\mu)$, $J_0 V(\theta_\mu) = V(\theta_\mu) J_0$, $V(\theta)V(\theta_\mu) = V(\theta_\mu)V(\theta)$, and $G V(\theta_\mu) = V(\theta_\mu) G$, we compute:
\begin{equation}\label{PHk-l}
	\begin{aligned}
		\overline{P^{\mu}} \# H^{\times (k-l)} 
		&= -\frac{1}{2} \left\langle J_0 V(\theta_\mu) e^{V(\theta_\mu)t} z, e^{V(\theta_\mu)t} z \right\rangle \\
		&\quad + \frac{1}{2}(k-l) \left\langle J_0 (V(\theta) + G) e^{V(\theta_\mu)t} z, e^{V(\theta_\mu)t} z \right\rangle + h^{\times (k-l)}(\phi^t_{P^\mu} z) \\
		&= -\frac{1}{2} \left\langle J_0 V(\theta_\mu) z, z \right\rangle + \frac{1}{2}(k-l) \left\langle J_0 (V(\theta) + G) z, z \right\rangle + h^{\times (k-l)}(\phi^t_{P^\mu} z).
	\end{aligned}
\end{equation}
Thus, the quadratic form $\overline{P^{\mu}} \# Q^{\times (k-l)}$ is time-independent. 
Moreover, 
\[
\phi^t_{\overline{P^{\mu}} \# Q^{\times (k-l)}} = e^{-V(\theta_\mu)t} e^{(k-l) m^j t},
\]
and so
\[
\phi^t_{\overline{P^{\mu}} \# Q^{\times (k-l)}} \phi^1_{Q^{\times l}} = e^{(k-l)G t} e^{-V(\theta_\mu)t + (k-l)V(\theta)t + l m^j}.
\]
This path is non-degenerate for all $t \in [0,1]$, as the eigenvalues of $\phi^t_{\overline{P^{\mu}} \# Q^{\times (k-l)}} \phi^1_{Q^{\times l}}$ are the same as $e^{-V(\theta_\mu)t + (k-l)V(\theta)t + l m^j}$ by \eqref{eq:eigVG}. The eigenvalues are
\[
\pm i \left[ l\theta + t \left( k\theta - 2\pi\,\mathrm{sgn}(\theta) \left\lfloor \frac{k|\theta|}{2\pi} \right\rfloor - \left( l\theta - 2\pi\,\mathrm{sgn}(\theta) \left\lfloor \frac{l|\theta|}{2\pi} \right\rfloor \right) \right) \right].
\]
Using the identity $k\theta = 2\pi\left( \mathrm{sgn}(\theta) \left\lfloor \frac{k|\theta|}{2\pi} \right\rfloor + \left\{ \frac{k\theta}{2\pi} \right\} \right)$ , where $\left\{ \cdot\right\}$ stands for the decimal portion, this simplifies to
\[
\pm i \left[ l\theta + 2\pi t \left( \left\{ \frac{k\theta}{2\pi} \right\} - \left\{ \frac{l\theta}{2\pi} \right\} \right) \right].
\]
If $k > l$ are sufficiently large primes such that $k\theta, l\theta \not\equiv 0 \pmod{2\pi}$, then the path$t \mapsto e^{i \left[ l\theta + 2\pi t \left( \left\{ \frac{k\theta}{2\pi} \right\} - \left\{ \frac{l\theta}{2\pi} \right\} \right) \right]}$ for $t \in [0,1]$ connects $e^{i l\theta}$ to $e^{i k\theta}$ without passing through $1$, ensuring non-degeneracy.

Case 4: Suppose $\phi^1_Q = e^{m^j}$, where $m^j = V(\theta) + G$ with $-\pi < \theta < \pi$, $\theta \ne 0$, is a $2(t_j+1) \times 2(t_j+1)$ matrix defined by \eqref{even mt} and $t_j$ is even. Then $i_\infty(H) = -\mathrm{sgn}(\epsilon \theta)$. For a fixed $\theta$, there exists a prime $p$ such that for all primes $l > p$, we have $l\theta \not\equiv 0 \pmod{2\pi}$, ensuring $\phi^l_Q$ is non-degenerate. Moreover, for any prime $l > p$,
\[
i_\infty(H,l) = -\mathrm{sgn}(\epsilon \theta) \left( 2\left\lfloor \frac{l|\theta|}{2\pi} \right\rfloor + 1 \right),
\]
and hence
\[
2\mu = i_\infty(H,k) - i_\infty(H,l) = -\mathrm{sgn}(\epsilon \theta) \left( 2\left\lfloor \frac{k|\theta|}{2\pi} \right\rfloor - 2\left\lfloor \frac{l|\theta|}{2\pi} \right\rfloor \right).
\]  

Define
\begin{equation}\label{Pmu4}
	P^\mu = \frac{1}{2} \left\langle J_0 V(\theta_\mu) z, z \right\rangle,
\end{equation}
where $\theta_\mu = -2\pi \mu$. Then $\phi^t_{P^\mu} = e^{V(\theta_\mu)t}$ satisfies $\phi^0_{P^\mu} = \phi^1_{P^\mu} = I_{2(t_j+1)}$, and $ \phi^t_{P^\mu}$ is a continuous path of unitary matrices with Maslov index $\mu$.

The identity \eqref{PHk-l} from Case 3 continues to hold, so the quadratic form $\overline{P^{\mu}} \# Q^{\times (k-l)}$ is time-independent. Furthermore,
\[
\phi^t_{\overline{P^{\mu}} \# Q^{\times (k-l)}} \phi^1_{Q^{\times l}} = e^{-V(\theta_\mu)t + (k-l) m^j t + l m^j} = e^{(k-l)G t} e^{-V(\theta_\mu)t + (k-l)V(\theta)t + l m^j}.
\]
As in Case 3, this path is non-degenerate for all $t \in [0,1]$.
\end{proof}

Based on Lemma \ref{Qklit}, we define the Hamiltonian
\begin{equation}
 	H^{k \odot l}_t(z) =\eta_{S_0}\left(\overline{P^{\mu}}\# Q^{\times (k-l)}\right)\wedge H^{\times l}=\eta_{S_0}\left(\overline{P^{\mu}}\# Q^{\times (k-l)}\right)\wedge Q^{\times l}+0\wedge h ^{\times l},  
 \end{equation}
where $\eta_{S_0}(t): \mathbb{R} \to \mathbb{R}$ is an odd function satisfying
\begin{equation}
	\eta_{S_0}(t) = 
	\begin{cases}
		0, & 0 \le t \le S_0, \\
		t - (S_0 + 1), & t \ge S_0 + 2,
	\end{cases}
\end{equation}
with $\eta'_{S_0} \in [0,1]$ and $\eta'_{S_0}$ monotonically non-decreasing. Moreover, there exists such a function $\eta_{S_0}$ satisfying the stated conditions, and $\lvert t - \eta_{S_0}(t) \rvert \le S_0 + 2$ for all $t \in \mathbb{R}$. Besides we take $$S_0=\max \limits_{\vert z\vert \le R_0}\,\left(\left \vert \overline{P^{\mu}}\# Q^{\times (k-l)}\right \vert\right)+1,$$ where $R_0$ satisfies the condition that $h \equiv 0$ when $|z| \ge R_0$. Furthermore, assuming that $X_H$ has only finitely many periodic solutions, it is possible to slightly adjust $R_0$ so that all periodic solutions are contained within $|z| < R_0$. Additionally, $R_0$ is chosen to be fixed. According to the proof of Lemma \ref{Qklit}, we can deduce that $S_0 = O(k - l)$.
 
Since $\overline{P^{\mu}} \# Q^{\times (k-l)}$ is a time-independent quadratic form, we may assume the existence of a real symmetric matrix $\hat{B}$ such that $\overline{P^{\mu}} \# Q^{\times (k-l)} = \frac{1}{2} \langle \hat{B} z, z \rangle$. The associated Hamiltonian vector field is given by  
$$
X_{\eta_{S_0}\left(\overline{P^{\mu}}\# Q^{\times (k-l)}\right)} = -J_0 \eta_{S_0}'\left( \frac{1}{2} \langle \hat{B} z, z \rangle \right) \hat{B} z,
$$  
and the corresponding flow is  
\begin{equation}
	\phi^t_{\eta_{S_0}\left(\overline{P^{\mu}}\# Q^{\times (k-l)}\right)}(z_0) = e^{-J_0 \eta'_{S_0}(H_0) \hat{B} t} z_0 = \phi^{\eta'_{S_0}(H_0)t}_{\overline{P^{\mu}}\# Q^{\times (k-l)}}(z_0),
\end{equation}  
where $H_0 = \frac{1}{2} \langle \hat{B} z_0, z_0 \rangle$.

In what follows, we aim to show that the one-periodic solutions of $0 \wedge H^{\times l}$ and $H^{k \odot l} = \eta_{S_0}\left(\overline{P^{\mu}}\# Q^{\times (k-l)}\right) \wedge H^{\times l}$ are identical. 	The proof consists of two parts:

Step 1: Interior one-periodic solutions.
	For any one-periodic solution $z(t)$ with $|z(t)| < R_0$, we have 
	$\left| \frac{1}{2} \langle \hat{B} z, z \rangle \right| < S_0 $, 
	so $\eta_{S_0} \left( \frac{1}{2} \langle \hat{B} z, z \rangle \right) = 0$. 
	Thus, $0 \wedge H^{\times l}$ and $H^{k \odot l} $  coincide on these solutions, and their Hamiltonian vector fields agree. 
	Therefore, the one-periodic solutions in $\overline{B}(R_0)$ , defined as the closed ball of radius $R_0$,  are identical.
	
	Step 2: No exterior one-periodic solutions.	Since $\overline{P^{\mu}} \# Q^{\times (k-l)}$ is autonomous, the value 
	$\frac{1}{2} \langle \hat{B} z, z \rangle$ is conserved along its flow.
If a given initial value $z_0$ satisfies $\left| \frac{1}{2} \langle \hat{B} z_0, z_0 \rangle \right| \ge S_0$, then  
	$$
	\left| \frac{1}{2} \left \langle \hat{B} \phi^{t}_{\overline{P^{\mu}} \# Q^{\times (k-l)}}(z_0), \phi^{t}_{\overline{P^{\mu}} \# Q^{\times (k-l)}}(z_0) \right \rangle \right| = \left| \frac{1}{2} \langle \hat{B} z_0, z_0 \rangle \right| \ge S_0
	$$
	for all $t \in \mathbb{R}$. Consequently, the trajectory  
	$$
	\phi^{t}_{\overline{P^{\mu}} \# Q^{\times (k-l)}}(z_0) = (\phi^t_{P^\mu})^{-1} \circ \phi^{(k-l)t}_{Q} (z_0)
	$$
	never enters the region $\overline{B}(R_0)$ for any $t \in \mathbb{R}$.  
	
Since $(\phi^t_{P^\mu})^{-1}$ is unitary and norm-preserving,
	$$
	\left| (\phi^t_{P^\mu})^{-1} \circ \phi^{(k-l)t}_{Q} (z_0) \right| = \left| \phi^{(k-l)t}_{Q} (z_0) \right|.
	$$
	From \eqref{defi Q1} and \eqref{defi Q2}, we have $\phi^t_Q(z_0) = e^{\hat{m}t}z_0$ or $e^{\pi J_0 t} e^{\hat{m}t} z_0$. Since $e^{\pi J_0 t}$ is a path of unitary matrices, it follows that $|e^{\pi J_0 t} e^{\hat{m}t} z_0| = |e^{\hat{m}t} z_0|$. Therefore, $e^{\hat{m}t} z_0$ remains outside $\overline{B}(R_0)$ for all $t \in \mathbb{R}$.
	
It follows from the above that if  $\phi^1_{H^{\times l}}(z_0)$ satisfies  
\begin{equation}\label{eq:hlS0}
	\left| \frac{1}{2} \langle \hat{B} \phi^1_{H^{\times l}}(z_0), \phi^1_{H^{\times l}}(z_0) \rangle \right| \ge S_0,
\end{equation}
then the trajectory $e^{\hat{m}t} \phi^1_{H^{\times l}}(z_0)$ remains outside $\overline{B}(R_0)$ for all $t \in \mathbb{R}$. In particular, the point $\phi^1_{H^{\times l}}(z_0)$ lies outside the region $\overline{B}(R_0)$.

We now show that if $e^{\hat{m}t}\phi^1_{H^{\times l}}(z_0) $  lies outside the region  $\overline{B}(R_0)$ for all $t\in \bR$, then $\phi^t_{H^{\times l}}(z_0)$ remains outside the region  $\overline{B}(R_0)$ for all $t \in [0, 1]$. We proceed by contradiction. Suppose there exists $t_1 \in [0, 1)$ such that $z(t_1)=\phi^{t_1}_{H^{\times l}}(z_0) \in \partial B(R_0)$, where $\partial B(R_0)$ denotes the sphere of radius $R_0$, and such that $\phi^{t}_{H^{\times l}}(z_0)$ lies outside the region $\overline{B}(R_0)$ for all $t \in (t_1, 1]$. 	
	 Since $H_t = Q_t$ outside $B(R_0)$, 
	the flow is linear. There are two cases:
	
	\begin{itemize}
		\item 
		If $\phi^1_Q = e^{\hat{m}}$, the quadratic form of the Hamiltonian function $H_t$ at infinity coincides with 
		\begin{equation}
			Q(z) = \frac{1}{2}\left\langle J_0\hat{m}z, z\right\rangle.
		\end{equation} 
		The associated linear Hamiltonian system is autonomous. For any $t_1$, the solution starting from $z(t_1)$ at time $t_1$ is 
		\begin{equation}\label{eq:zet1}
			z(t)=e^{(t-t_1)\hat{m}}z(t_1)	.     
		\end{equation}
 then 
		$\phi^1_{H^{\times l}}(z_0) = e^{l\hat{m}(1-t_1)} z(t_1)$, 
		so $e^{\hat{m}t} \phi^1_{H^{\times l}}(z_0) = e^{\hat{m}t} e^{l\hat{m}(1-t_1)} z(t_1)$. 
		Setting $t = -l(1-t_1)$ yields $e^{\hat{m}t} \phi^1_{H^{\times l}}(z_0) = z(t_1) \in \partial B(R_0)$, 
		a contradiction.
		
		\item If $\phi^1_Q = -e^{\hat{m}}$,  the quadratic form of the Hamiltonian function $H_t$ at infinity coincides with 
		\begin{equation}
			Q_t(z) = \frac{1}{2}\left\langle -\pi I_{2n}z, z\right\rangle + \frac{1}{2}\left\langle J_0\hat{m}e^{-\pi J_0 t}z, e^{-\pi J_0 t}z\right\rangle.    
		\end{equation}
		The Hamiltonian system associated with the Hamiltonian function $Q_t$ is given by
		\begin{equation}
			\dot{z} = \pi J_0 z + e^{\pi J_0 t} \hat{m} e^{-\pi J_0 t} z.
		\end{equation}
		The solution starting from $z(t_1)$ at time $t_1$ is expressed as
		\begin{equation}\label{eq:zet2}
			z(t) = e^{\pi J_0 t} e^{\hat{m}(t - t_1)} e^{-\pi J_0 t_1} z(t_1).
		\end{equation}
		 then 
		   \begin{equation}\label{eq:emHl2}
		 	\phi^1_{H^{\times l}}(z_0)=e^{\pi J_0 l} e^{\hat{m}l(1 - t_1)} e^{-\pi J_0 lt_1} z(t_1)=- e^{\hat{m}l(1 - t_1)} e^{-\pi J_0 lt_1} z(t_1),
		 \end{equation}
		 where $l$ is a prime number and $e^{\pi J_0l}=-I_{2n}$ holds.	So $$e^{\hat{m}t} \phi^1_{H^{\times l}}(z_0) = - e^{\hat{m}t} e^{\hat{m}l(1-t_1)} e^{-\pi J_0 l t_1} z(t_1).$$ 
		Setting $t = -l(1-t_1)$ gives $e^{\hat{m}t} \phi^1_{H^{\times l}}(z_0) = - e^{-\pi J_0 l t_1} z(t_1)$. 
		As $e^{-\pi J_0 l t_1}$ is unitary, this point lies on $\partial B(R_0)$, a contradiction.
	\end{itemize}
	
 From the above analysis, it follows that if  $\phi^1_{H^{\times l}}(z_0)$ satisfies condition \eqref{eq:hlS0}, then $\phi^t_{H^{\times l}}(z_0)$ remains outside the region $\overline{B}(R_0)$ for all $t \in [0,1]$. Hence, $\phi^l_H(z_0) = \phi^l_Q(z_0)$.
 
	The flow of $H^{k \odot l}$ is given by:
 \begin{equation}
 	\phi^t_{H^{k \odot l}}(z_0)=\begin{cases}
 		\phi^{\rho(2t)}_{H^{\times l}}(z_0), & t\in [0,\frac{1}{2}] \\
 		e^{-J_0\eta'_{S_0}(H_0)\hat{B}\rho(2t-1)}\phi^1_{H^{\times l}}(z_0), & t\in [\frac{1}{2},1] \\
 	\end{cases},
 \end{equation}
 where $H_0=\frac{1}{2}\left\langle \hat{B}\phi^l_H(z_0),\phi^l_H(z_0)\right\rangle $. 
  It can be known that when  $\phi^1_{H^{\times l}}(z_0)$ satisfies condition  \eqref{eq:hlS0}, then we have
 \begin{equation}
 	\phi^t_{H^{l\odot k}}(z_0)=\phi^{\eta'(H_0)t}_{\overline{P^{\mu}}\# Q^{\times (k-l)}}\circ \phi^l_Q(z_0),\quad t\in [\frac{1}{2},1],
 \end{equation}
 where $H_0=\frac{1}{2}\left\langle \hat{B}\phi^l_Q(z_0),\phi^l_Q(z_0)\right\rangle $. By Lemma \ref{Qklit}, $\phi^1_{H^{l\odot k}}(z_0)$ is non-degenerate, which ensures that no new one-periodic solutions arise in this case. Furthermore, if $\phi^1_{H^{\times l}}(z_0)$ does not satisfy condition \eqref{eq:hlS0}, then $\eta'=0$, and thus 
 \begin{equation}
 	\phi^t_{H^{l\odot k}}(z_0)= \phi^l_H(z_0),\quad t\in [\frac{1}{2},1].
 \end{equation}
 Therefore, no new one-periodic solutions arise.
 	
	Therefore, the one-periodic solutions of  $0 \wedge H^{\times l}$ and $H^{k \odot l} $  are identical.

It can be verified that the above procedure remains valid when $\phi^1_{Q} = M_1 \diamond M_2$, where $M_1 \in \Sp(2i)$ and $M_2 \in \Sp(2j)$, and there exist $\hat{m}_1$ and $\hat{m}_2$ such that $M_1 = e^{\hat{m}_1}$ and $M_2 = -e^{\hat{m}_2}$. Under these conditions, it follows that the one-periodic solutions of $0 \wedge H^{\times l}$ are also one-periodic solutions of $H^{l \odot k}$.

Moreover,  recall that  
\begin{equation}  
	H^{k \ominus l} = \left(\overline{P^\mu} \# H^{\times (k-l)}\right) \wedge H^{\times l} =\left( \overline{P^{\mu}} \# Q^{\times (k-l)} \right)\wedge Q^{\times l} + h^{\times (k-l)}(\phi^t_{P^\mu} z) \wedge h^{\times l},  
\end{equation}  
so we have 
\begin{equation}\label{eq:esH}
	\lVert H^{k \odot l} - H^{k \ominus l} \rVert_{L^\infty} \le S_0 + 2 + (k-l) \lVert h \rVert_{L^\infty}.	
\end{equation} 
Since $S_0 = O(k-l)$, it follows that  
$$  
\lVert H^{k \odot l} - H^{k \ominus l} \rVert_{L^\infty} = O(k-l).  
$$

\section{The proof of Theorem}
\label{The proof of Theorem}

\subsection{The well-definedness of maps between Floer homologies}
\label{The well-definedness of maps between Floer homologies}
Assume that $H_0=H^{k\odot l}$ and $H_1 =H^{k\ominus l}$. In this section, we first show that the map 
\begin{equation}
	\Psi_{H_0,H_1}:HF^{[a,b)}(H_0)\rightarrow HF^{[a+C,b+C)}(H_1).
\end{equation} 
between Floer homologies is well-defined.  

\begin{lemma}\label{Le:Bounded}	
Let $K = H^{k \ominus l}$ or $H^{k \odot l}$. Then $K$ satisfies conditions \ref{Linear} Linear growth of the Hamiltonian vector field and \ref{nonres} Nonresonance at infinity.
\end{lemma}	
\begin{proof}
	
	We denote $\hat{K} = \left(\overline{P^{\mu}}\# Q^{\times (k-l)}\right)\wedge Q^{\times l}$ or $\eta_{S_0}\left(\overline{P^{\mu}}\# Q^{\times (k-l)}\right)\wedge Q^{\times l}$, which represents the quadratic form of $H^{k\ominus l}$ or $H^{k\odot l}$ at i nfinity. We still denote that there exists a real symmetric matrix $\hat{B}$ such that $\overline{P^{\mu}}\# Q^{\times (k-l)}=\frac{1}{2}\left\langle \hat{B}z,z\right\rangle $. So we have
	
	\begin{equation}
		\nabla \hat{K} =
		\begin{cases}
			2\rho'(2t)\nabla Q^{\times l}, & t \in [0,\frac{1}{2}] \\
			2\rho'(2t-1)\hat{B}z, & t \in [\frac{1}{2},1]
		\end{cases} \,\,\, \text{or}\,\,\, 	\nabla \hat{K} =
		\begin{cases}
		2\rho'(2t)\nabla Q^{\times l}, & t \in [0,\frac{1}{2}] \\
		2\rho'(2t-1)\eta'_{s_0}(H_0)\hat{B}z, & t \in [\frac{1}{2},1]
		\end{cases},
	\end{equation}
	where $H_0 = \frac{1}{2}\langle \hat{B}z, z \rangle$.
	
	As $Q^{\times l}$ and $\overline{P^{\mu}}\# Q^{\times (k-l)}$ are quadratic forms, 
	$\nabla Q^{\times l}$ and $\hat{B} z$ are linear in $z$. 
	Since the coefficients $\rho'$ and $\eta'_{S_0}$ are bounded, 
	there exists $c_1 > 0$ such that 
	$$|X_{\hat{K}}| = |-J_0\nabla \hat{K}| \le c_1|z|.$$
	Moreover, $K - \hat{K}$ has compact support, so
	there exists $C_1 > 0$ such that $|X_{K - \hat{K}}| \le C_1$ for all $(t,z)$.
 Therefore,
	\begin{equation}
		|X_K| \le |X_{\hat{K}}| + |X_{K - \hat{K}}| \le \max\{c_1, C_1\}(|z| + 1),
	\end{equation}
which shows that $K$ satisfies condition \ref{Linear}.

The Hamiltonian system associated with $\hat{K}$ is given by
\begin{equation}\label{eq:Ha sys}
	\dot{z} = X^t_{\hat{K}}(z) = -J_0\nabla \hat{K}.
\end{equation}
The fundamental solution matrix for this system with initial value $z_0$ takes the form:
\begin{equation}
	\Phi(t) = 
	\begin{cases}
		\phi^{\rho(2t)}_{Q^{\times l}}, & t \in [0,\frac{1}{2}]\\
		e^{-J_0\hat{B}\rho(2t-1)}\phi^{1}_{Q^{\times l}}, & t \in [\frac{1}{2},1]
	\end{cases}
\end{equation}
or 
\begin{equation}
	\Phi(t,z_0) = 
	\begin{cases}
		\phi^{\rho(2t)}_{Q^{\times l}}, & t \in [0,\frac{1}{2}] \\
		e^{-J_0\eta'_{s_0}(H_0)\hat{B}\rho(2t-1)}\phi^{1}_{Q^{\times l}}, & t \in [\frac{1}{2},1]
	\end{cases},
\end{equation}
where $H_0 = \frac{1}{2}\langle \hat{B}\phi^{1}_{Q^{\times l}}z_0, \phi^{1}_{Q^{\times l}}z_0 \rangle$. The dependence of $\Phi(t,z_0)$ on $z_0$ occurs only through the parameter $s = \eta'_{S_0}(H_0) \in [0,1]$. Thus, we can parameterize $\Phi(t,z_0)$ as $\Phi(t,s)$ with $(t,s) \in [0,1] \times [0,1]$. Since $\Phi(t,s)$ is continuous on this compact set, the following constants are finite:
\begin{equation}
	M_1 = \sup_{\substack{t \in [0,1] \\ z_0 \in \mathbb{R}^{2n}}} \| \Phi(t,z_0) \| < \infty, \quad
	M_2 = \sup_{\substack{t \in [0,1] \\ z_0 \in \mathbb{R}^{2n}}} \| \Phi^{-1}(t,z_0) \| < \infty, 
\end{equation}
\begin{equation}
		C_2 = \sup_{z_0 \in \mathbb{R}^{2n}} \| (I - \Phi(1,z_0))^{-1} \Phi(1,z_0) \| < \infty.
\end{equation}
The finiteness of $C_2$ follows from Lemma \ref{Qklit}, which ensures that $I - \Phi(1,z_0)$ is invertible for all $z_0$, and the continuity of the matrix inversion on the compact parameter space. 

The first fundamental solution matrix described above is independent of the initial value \(z_0\). For notational convenience, we denote it by $\Phi(t,z_0)$.

Now consider the perturbed Hamiltonian system:
\begin{equation} \label{eq:PHa}
	\dot{z} = X^t_K(z) + p(t) = X^t_{\hat{K}}(z) + X^t_{K-\hat{K}}(z) + p(t), \quad \| p \|_{L^2(S^1)} \le \epsilon.
\end{equation}
Define $e(t,z) = X^t_{K-\hat{K}}(z) + p(t)$. For a fixed initial value $z_0$, the solution satisfies the integral equation:
\begin{equation}
	z(t) = \Phi(t,z_0) z_0 + \int_0^t \Phi(t,z_0) \Phi^{-1}(s,z_0) e(s,z(s)) \, ds.
\end{equation}
Imposing the periodic condition $z(1) = z_0$ yields:
\begin{equation}
	z_0 = \Phi(1,z_0) z_0 + \int_0^1 \Phi(1,z_0) \Phi^{-1}(s,z_0) e(s,z(s)) \, ds.
\end{equation}
Rewriting this expression:
\begin{equation}
	(I - \Phi(1,z_0)) z_0 = \int_0^1 \Phi(1,z_0) \Phi^{-1}(s,z_0) e(s,z(s)) \, ds.
\end{equation}
Since $I - \Phi(1,z_0)$ is invertible by Lemma \ref{Qklit}, we obtain:
\begin{equation}
	z_0 = (I - \Phi(1,z_0))^{-1} \int_0^1 \Phi(1,z_0) \Phi^{-1}(s,z_0) e(s,z(s)) \, ds.
\end{equation}
Taking norms and using the boundedness of the operators:
\begin{equation}
	|z_0| \le C_2 M_2 \int_0^1 |e(s,z(s))| \, ds.
\end{equation}
Since $K - \hat{K}$ has compact support, there exists $C_1 > 0$ such that $|X^t_{K-\hat{K}}(z)| \le C_1$ for all $(t,z)$. Combined with the bound on $p(t)$, we have:
\begin{equation}
	\|e(s,z(s))\|_{L^1(S^1)} \le \|X^t_{K-\hat{K}}(z)\|_{L^1(S^1)} + \|p\|_{L^1(S^1)} \le C_1 + \epsilon.
\end{equation}
Thus:
\begin{equation}
	|z_0| \le C_2 M_2 (C_1 + \epsilon).
\end{equation}

To estimate the $L^2$-norm of $z(t)$, we use:
\begin{equation}
	\|z\|_{L^2(S^1)} \le M_1 |z_0| + M_1 M_2 \left\| \int_0^t e(s,z(s)) \, ds \right\|_{L^2(S^1)}.
\end{equation}
For the second term,  we have 
\begin{equation}
	\left\| \int_{0}^{t} e(s,z(s))ds \right\|^2_{L^2(S^1)}=\int_{0}^{1}\left(\int_{0}^{t}|e(s,z(s))|ds\right)^2dt \le \int_{0}^{1}t \int_{0}^{t} |e(s,z(s))|^2dsdt .
\end{equation}
Since $1-s^2\le 1$ for all $s\in [0,1]$, interchanging the order of integration yields 
\begin{equation}  
	\begin{aligned}
		\int_{0}^{1}t \int_{0}^{t} |e(s,z(s))|^2dsdt &=	\int_{0}^{1}|e(s,z(s))|^2\left(\int_{0}^{1}tdt\right)ds
		= \frac{1}{2}\int_{0}^{1}|e(s,z(s))|^2(1-s^2)ds \\  & \le \frac{1}{2}\int_{0}^{1}|e(s,z(s))|^2ds=	 \frac{1}{2} \|e(s,z(s))	\|^2_{L^2(S^1)}	.
	\end{aligned}
\end{equation}
Therefore:
\begin{equation}
	\left\| \int_0^t e(s,z(s)) \, ds \right\|_{L^2(S^1)} \le \frac{1}{\sqrt{2}} \|e(s,z(s))\|_{L^2(S^1)} \le \frac{1}{\sqrt{2}} (C_1 + \epsilon).
\end{equation}
Combining all estimates:
\begin{equation}
	\|z\|_{L^2(S^1)} \le M_1 C_2 M_2 (C_1 + \epsilon) + M_1 M_2 \frac{1}{\sqrt{2}} (C_1 + \epsilon) = M_1 M_2 \left( C_2 + \frac{1}{\sqrt{2}} \right) (C_1 + \epsilon).
\end{equation}

This establishes that for any $\epsilon > 0$, if $\|\dot{z} - X_K(z)\|_{L^2(S^1)} \le \epsilon$, then $\|z\|_{L^2(S^1)} \le M_1 M_2 \left( C_2 + \frac{1}{\sqrt{2}} \right) (C_1 + \epsilon)$. Therefore, $K$ satisfies Condition \ref{nonres}.
\end{proof}

Since both $H^{k\ominus l}$ and $H^{k\odot l}$ satisfy condition \ref{Linear}, there exist constants $c_1$ and $c_2$ such that
\begin{equation}
	|X_{H^{k\ominus l}}| \le c_1(1 + |z|) \quad \text{and} \quad |X_{H^{k\odot l}}| \le c_2(1 + |z|).
\end{equation}

Define the homotopy
\begin{equation}
	H^s_t(z) = \left(1 - f(s)\right) H^{k\odot l} + f(s) H^{k\ominus l},
\end{equation}
where $f: \mathbb{R} \to [0,1]$ is a smooth monotonically increasing function satisfying $f(s) = 0$ for $s \le 0$, $f(s) = 1$ for $s \ge 1$, and $f' < 2$.
For all $(s,t) \in \mathbb{R} \times S^1$, the Hamiltonian vector field satisfies
\begin{equation}
	\begin{aligned}
		|X_{H^s_t}| &= \left|\left(1 - f(s)\right) X_{H^{k\odot l}} + f(s) X_{H^{k\ominus l}}\right| \\
		&\le \left(1 - f(s)\right) c_1(1 + |z|) + f(s) c_2(1 + |z|) \\
		&\le \max\{c_1, c_2\}(1 + |z|),
	\end{aligned}
\end{equation}
which implies that $H^s_t$ satisfies condition \ref{Linear} uniformly in $(s,t)$.

Moreover, we have the estimate
\begin{equation}
	\begin{aligned}
		\int_{-\infty}^{\infty} \int_{S^1} \max_{z \in \mathbb{R}^{2n}} \partial_s H^s_t(z) \, dt \, ds 
		&= \int_{-\infty}^{\infty} \int_{S^1} \max_{z \in \mathbb{R}^{2n}} f'(s)(H^{k\ominus l} - H^{k\odot l}) \, dt \, ds \\
		&\le 2 \| H^{k\ominus l} - H^{k\odot l} \|_{L^\infty}.
	\end{aligned}
\end{equation}
By \eqref{eq:esH}, for any interval $[a, b)$, there exists a homomorphism
\begin{equation}\label{eq:H_0 a H_1}
	\Psi_{H^{k\odot l}, H^{k\ominus l}}: HF^{[a,b)}(H^{k\odot l}) \rightarrow HF^{[a+C,b+C)}(H^{k\ominus l}),
\end{equation}
where $C = 2\| H^{k\ominus l} - H^{k\odot l} \|_{L^\infty}$.

Assume that the Hamiltonian diffeomorphism $\phi^1_H$ has only finitely many fixed points, all of which are isolated. When the prime numbers $k>l$ are sufficiently large, it can be ensured that the prime number is admissible for each fixed point. 
Let $z_0 \in \operatorname{Fix}(\phi^1_H) = \operatorname{Fix}(H^{k\ominus l})$ with $\mathcal{A}(z_0) = a_0$, where $a_0$ is an isolated critical value. It then follows from \cite{GVBZ2010} and \cite{ML2024} that the Floer homology groups are isomorphic up to a degree shift. By Lemma \ref{le:PandH} and the definition of local Floer homology, we have
\begin{equation}\label{HFloc}
	HF^{\mathrm{loc}}_*(H^{k\ominus l}, z_0) = HF^{\mathrm{loc}}_{*+2\mu}(H^{\times k}, z_0),
\end{equation}
where $2\mu = i_\infty(H^{\times k}) - i_\infty(H^{\times l})$.
Furthermore, for sufficiently small $\epsilon > 0$,
\begin{equation}
	HF^{[a_0 - \epsilon, a_0 + \epsilon)}_*(H^{k\ominus l}) = HF^{[a_0 - \epsilon, a_0 + \epsilon)}_{*+2\mu}(H^{\times k}).
\end{equation}

Furthermore, compared to the function $0\wedge H^{\times l}$, the function $H^{k\odot l}$ only has an additional  "tail" at infinity. In the region where one-periodic solutions exist, $0\wedge H^{\times l}$ and  $H^{k\odot l}$ coincide. Consequently, their local Floer homologies are isomorphic,  $$ HF_*^\mathrm{loc}( H^{\times l},z_0) = HF_*^\mathrm{loc}(0\wedge H^{\times l},z_0) = HF_*^\mathrm{loc}(H^{k\odot l},z_0) .$$

\subsection{The proof of Theorem 1}
\label{The proof of Theorem 1}

\begin{myproof}
Assume by contradiction that $\phi^1_H$ has only finitely many fixed points and finitely many simple periodic points that are not iterations of points with smaller period. 
Consider any sufficiently large increasing sequence of prime numbers $\{p_i\}$. Such a sequence can be chosen to be admissible with respect to both $\phi^1_H$ and the non-degenerate quadratic form $Q_t$, while satisfying the growth condition $p_{i+1} - p_i = o(p_i)$.	For details, see Reference \cite{BHP2001}.
	
For any indices $i$ and $m$, the difference $p_{i+m} - p_i$ can be expressed as a telescoping sum of consecutive prime differences, yielding $p_{i+m} - p_i = o(p_i)$. Under iterations along this sequence, the fixed points of $\phi^1_H$ remain isolated. Moreover, by \cite{GVBZ2010}, the local Floer homology groups of $H$ and $H^{\times p_j}$ coincide up to a degree shift. Additionally, the quadratic form of $H^{\times p_i}$ at infinity remains non-degenerate.

Let $a_0 = \mathcal{A}_H(z_0)$ denote the action value of the one-periodic solution through $z_0$. Since $\phi^1_H$ has only finitely many fixed points, $a_0$ is an isolated point in the action spectrum $\mathscr{L}(H)$, there exists $\epsilon_0 > 0$ such that 
\begin{equation}
	[a_0 - \epsilon_0, a_0 + \epsilon_0) \cap \mathscr{L}(H) = \{a_0\}.
\end{equation}	
Assume further that the initial prime $p_1$ exceeds the period of any periodic point of $X_H$. Then all one-periodic solutions of $H^{p_{j+m} \odot p_j}$ arise as $p_j$-fold iterations of one-periodic solutions of $X_H$, and we have
\begin{equation}
	\mathcal{A}_{H^{p_{j+m} \odot p_j}}(z_0) = p_j \mathcal{A}_H(z_0).
\end{equation}
Consequently,
\begin{equation}
	[p_j a_0 - p_j \epsilon_0, p_j a_0 + p_j \epsilon_0) \cap \mathscr{L}(H^{p_{j+m} \odot p_j}) = \{p_j a_0\}.
\end{equation}	
	
Let $C_{j,m} = 2\|H^{p_{j+m} \odot p_j} - H^{p_{j+m} \ominus p_j}\|_{L^\infty}$. 
Since $C_{j,m} = O(p_{j+m} - p_j) = o(p_j)$, it follows that for the previously chosen $\epsilon_0 > 0$, there exists $j_0$ such that for all $j > j_0$, 
\begin{equation}
	\frac{C_{j,m}}{p_j} < \frac{\epsilon_0}{6}.
\end{equation}
Define the interval
\begin{equation}
	I = \left[p_j\left(a_0 - \frac{\epsilon_0}{3}\right), p_j\left(a_0 + \frac{\epsilon_0}{3}\right)\right).
\end{equation}

Then we have
\begin{equation}
	I \cap (I + C_{j,m}) \cap (I + 2C_{j,m}) \cap \mathscr{L}(H^{p_{j+m} \odot p_j}) = \{p_j a_0\},
\end{equation}
and
\begin{equation}
	I \cup (I + C_{j,m}) \cup (I + 2C_{j,m}) \subset [p_j a_0 - p_j \epsilon_0, p_j a_0 + p_j \epsilon_0).
\end{equation}	
	
According to \eqref{eq:H_0 a H_1}, there exist homomorphisms
\begin{equation}
	\Psi_{H^{p_{j+m}\odot p_j},H^{p_{j+m}\ominus p_j}}:HF^{I}(H^{p_{j+m}\odot p_j})\rightarrow HF^{I+C_{j,m}}(H^{p_{j+m}\ominus p_j}),
\end{equation}
\begin{equation}
	\Psi_{H^{p_{j+m}\ominus p_j},H^{p_{j+m}\odot p_j}}:HF^{I+C_{j,m}}(H^{p_{j+m}\ominus p_j})\rightarrow HF^{I+2C_{j,m}}(H^{p_{j+m}\odot p_j}).  
\end{equation}

Let $H^s$ be a $C_{j,m}$-bounded homotopy from $H^{p_{j+m} \odot p_j}$ to $H^{p_{j+m} \ominus p_j}$. 
Then $H^{-s}$ is a $C_{j,m}$-bounded homotopy in the reverse direction. 
Define the composition $H^s \#_T H^{-s}$ for sufficiently large $T > 0$ by
\begin{equation}
	H^s \#_T H^{-s} = 
	\begin{cases}
		H^{s + T}, & s \leq 0, \\
		H^{-s - T}, & s \geq 0.
	\end{cases}
\end{equation}
This composition is clearly $2C_{j,m}$-bounded.	
The homotopy $H^s \#_T H^{-s}$ induces a map
\begin{equation}
	\overline{\Psi}_{H^{p_{j+m}\odot p_j},H^{p_{j+m}\odot p_j}}: HF^{I}(H^{p_{j+m} \odot p_j}) \rightarrow HF^{I + 2C_{j,m}}(H^{p_{j+m} \odot p_j}),
\end{equation}
which equals the composition$$\Psi_{H^{p_{j+m}\odot p_j},H^{p_{j+m}\ominus p_j}}\circ \Psi_{H^{p_{j+m}\ominus p_j},H^{p_{j+m}\odot p_j}}.$$

There exists a family of $2C_{j,m}$-bounded homotopies $H^{s,\lambda}$ ($\lambda \in [0,1]$) 
connecting $H^s \#_T H^{-s}$ to the identity homotopy $\mathcal{I}$, with 
$H^{s,0} = H^s \#_T H^{-s}$ and $H^{s,1} = \mathcal{I}$(\cite{GVL2007}). 
Since the induced map is independent of the choice within this family, we have
\begin{equation}
\overline{\Psi}_{H^{p_{j+m}\odot p_j},H^{p_{j+m}\odot p_j}}=\hat{\Psi}_{H^{p_{j+m}\odot p_j},H^{p_{j+m}\odot p_j}},
\end{equation}
where $\hat{\Psi}$ is the homomorphism induced by the identity homotopy.
The map $\hat{\Psi}$ factors as the composition
\begin{equation}
	HF^I(H^{p_{j+m} \odot p_j}) \rightarrow HF^{I + C_{j,m}}(H^{p_{j+m} \odot p_j}) \rightarrow HF^{I + 2C_{j,m}}(H^{p_{j+m} \odot p_j})
\end{equation}
of inclusion and quotient maps. In particular, $\hat{\Psi}$ acts as the identity on classes 
with action in $I \cap (I + 2C_{j,m})$ and sends all other classes to zero. 
Since $p_j a_0$ is the unique action value in $I \cap (I + 2C_{j,m})$, 
it follows that $\hat{\Psi}$ is the identity map. Thus, $\overline{\Psi}_{H^{p_{j+m}\odot p_j},H^{p_{j+m}\odot p_j}}$ is also the identity map.
	
Naturally, the commutativity property is evident in the following diagram:
	\begin{equation}\label{eq:crucial_triangle_abstract}
		\small{
	\begin{tikzcd}
		& \HF^{I+C_{j,m}}_*\left(H^{p_{j+m}\ominus p_j}\right)\ar[dr, left," \Psi_{H^{p_{j+m}\ominus p_j},H^{p_{j+m}\odot lp_j}} "]&\\
		\HF^{I}_*\left(H^{p_{j+m}\odot p_j}\right)\ar[rr,"\hat{\Psi}_{H^{p_{j+m}\odot p_j},H^{p_{j+m}\odot p_j}}=I"]\ar[ur, left,"\Psi_{H^{p_{j+m}\odot p_j},H^{p_{j+m}\ominus p_j}}"]& &\HF^{I+2C}_*\left(H^{p_{j+m}\odot p_j}\right)
	\end{tikzcd}
}
\end{equation}

Assume that the fixed point $z_0$ is isolated and homologically non-trivial. Then the local Floer homology $HF^{\text{loc}}_*(H^{p_{j+m} \odot p_j}, z_0)$ is non-trivial, and its support satisfies
\begin{equation}
	\text{supp} \, HF^{\text{loc}}_*(H^{p_{j+m} \odot p_j}, z_0) \subset \Delta(z_0, H^{p_{j+m} \odot p_j}) = [p_j \hat{i}_H(z_0) - n, p_j \hat{i}_H(z_0) + n].
\end{equation}
Furthermore, if $\hat{i}_H(z) \neq \hat{i}_H(z_0)$, then for sufficiently large $p_j$, $$[p_j\hat{i}_H(z_0)-n,p_j\hat{i}_H(z_0)-n]\cap [p_j\hat{i}_H(z)-n,p_j\hat{i}_H(z)-n]=\varnothing,$$
 Hence, for any $s \in \text{supp} \, HF^{\text{loc}}_*(H^{p_{j+m} \odot p_j}, z_0)$ and any $J \in [I, I+2C_{j,m}]$, we have
\begin{equation}\label{HFsuppz0}
	\{0\} \neq HF^J_s(H^{p_{j+m} \odot p_j}) \cong \bigoplus \left\{ HF^{\text{loc}}_s(H^{p_{j+m} \odot p_j}, z) : 
	\begin{matrix}
		z \in \Fix(\phi_{H^{p_{j+m} \odot p_j}}^1) \\
		\mathcal{A}_{H^{p_{j+m} \odot p_j}}(z) = p_j a_0 \\
		\hat{i}_H(z) = \hat{i}_H(z_0)
	\end{matrix} 
	\right\}.
\end{equation}

For any $z \in \Fix(\phi_H^1)$, equation \eqref{HFloc} implies
\begin{equation}\label{eq:degshift}
	HF^{\text{loc}}_*(H^{p_{j+m} \ominus p_j}, z) \cong HF^{\text{loc}}_{*+2\mu}(H^{\times p_{j+m}}, z),
\end{equation}
where $2\mu = i_\infty(H^{\times p_{j+m}}) - i_\infty(H^{\times p_j})$, and
\begin{equation}\label{eq:mu_eq}
	(p_{j+m} - p_j) \hat{i}_{\infty}(H) - n \le 2\mu \le (p_{j+m} - p_j) \hat{i}_{\infty}(H) + n.
\end{equation}
Consequently, the support satisfies
\begin{equation}
	\begin{aligned}
		\text{supp} \, HF^{\text{loc}}_*(H^{p_{j+m} \ominus p_j}, z) &\subset \Delta(z, H^{p_{j+m} \ominus p_j}) \\
		&= [p_{j+m} \hat{i}_H(z) - n - 2\mu, p_{j+m} \hat{i}_H(z) + n - 2\mu].
	\end{aligned}
\end{equation}

We now show that for any fixed point $z$ of $\phi_H^1$, if either $p_{j+m} - p_j$ or $p_j$ is sufficiently large, then 
\[
\Delta(z_0, H^{p_{j+m} \odot p_j}) \cap \Delta(z, H^{p_{j+m} \ominus p_j}) = \varnothing.
\]
Combining this with \eqref{HFsuppz0}, the homomorphism $\Psi_{H^{p_{j+m} \odot p_j}, H^{p_{j+m} \ominus p_j}}$ becomes zero, contradicting the commutativity of diagram \eqref{eq:crucial_triangle_abstract} and thus contradicts the original assumption.
Observe that
\[
\Delta(z_0, H^{p_{j+m} \odot p_j}) \cap \Delta(z, H^{p_{j+m} \ominus p_j}) = \varnothing \quad \Leftrightarrow \quad |p_{j+m} \hat{i}_H(z) - p_j \hat{i}_H(z_0) - 2\mu| > 2n.
\]
Using the estimate for $2\mu$ from \eqref{eq:mu_eq}, we deduce that if
\begin{equation}\label{eq:nodis}
	|p_{j+m} \hat{i}_H(z) - p_j \hat{i}_H(z_0) - (p_{j+m} - p_j) \hat{i}_{\infty}(H)| > 3n,
\end{equation}
then the intersection of the two intervals is empty. We consider two cases based on the mean index of $z$:

Case 1: $\hat{i}_H(z) = \hat{i}_H(z_0)$. Under the twist condition $\hat{i}_H(z_0) \neq \hat{i}_{\infty}(H)$, inequality \eqref{eq:nodis} simplifies to
\[
(p_{j+m} - p_j) |\hat{i}_H(z_0) - \hat{i}_{\infty}(H)| > 3n.
\]
Since $p_{j+m} - p_j > 2m$, we may choose $m$ sufficiently large to satisfy this inequality, ensuring the intervals are disjoint.

Case 2: $\hat{i}_H(z) \neq \hat{i}_H(z_0)$. The condition $p_{j+m} - p_j = o(p_j)$ implies$$\lim_{p_j\rightarrow +\infty}\frac{p_{j+m}}{p_j}=1,$$ for fixed $m$. Then
\[
\lim_{j \to \infty} \frac{1}{p_j} \left| p_{j+m} \hat{i}_H(z) - p_j \hat{i}_H(z_0) - (p_{j+m} - p_j) \hat{i}_{\infty}(H) \right| = |\hat{i}_H(z) - \hat{i}_H(z_0)| > 0.
\]
Thus, for sufficiently large $p_j$, inequality \eqref{eq:nodis} holds, and the intervals are disjoint.

\end{myproof}

\subsection{The proof of Theorem 2}
\label{The proof of Theorem 2}

\begin{myproof}
	
Let $H_t = Q_t + h_t$ be a Hamiltonian that equal to a non-degenerate quadratic form  $Q_t$ at infinity. Then the total Floer homology satisfies
\begin{equation}\label{totalfloer}
	HF_*(H) \cong HF_*(Q) =
	\begin{cases}
		\mathbb{Z}/2, & * = i_\infty(H), \\
		0, & \text{otherwise}.
	\end{cases}
\end{equation}

In Section 8.3 of Reference \cite{LY2012}, the iteration formula for the Conley–Zehnder index of paths in $\Sp(2n)$ is established. Note that the Conley–Zehnder index in the present work is the negative of that in \cite{LY2012}, so the mean indices also differ by a sign. In the convention of \cite{LY2012}, the mean index and Conley–Zehnder index of a non-degenerate symplectic path $\gamma(t)$ are related by
\begin{equation}\label{meanindex}
	\hat{i}(\gamma) = i(\gamma) - r + \sum_{j=1}^{r} \frac{\theta_j}{\pi}.
\end{equation}
where $r \leq n$, and $e^{i\theta_j}$ are certain eigenvalues of $\gamma(1)$ on the unit circle with $\theta_j \in (0,\pi) \cup (\pi, 2\pi)$. 

 Assume that $\phi^1_H$ has at least two non-degenerate fixed points. Without loss of generality, suppose that their mean indices are both equal to $\hat{i}_\infty(H)$. On the other hand, by applying Theorem \ref{theo1}, it follows that $\phi^1_H$ possesses simple periodic orbits with arbitrarily large prime periods.
	
For $n = 1$, the relationship between the mean index and Conley–Zehnder index of a non-degenerate fixed point $z$ of $\phi_H^1$ falls into one of two cases:
\begin{enumerate}
	\item $\hat{i}_H(z) = i_H(z)$;
	\item $\hat{i}_H(z) = i_H(z) + 1 - \dfrac{\theta}{\pi}$, where $\theta \in (0,\pi) \cup (\pi, 2\pi)$.
\end{enumerate}
In the first case, $i_H(z)$ may be odd or even. In the second case, $i_H(z)$ must be odd and $\hat{i}_H(z)$ is non-integral.   	
Now consider two scenarios based on the eigenvalues of $\phi_Q^1$.

If the eigenvalues of $\phi_Q^1$ are real eigenvalues distinct from $1$. Here, $i_\infty(H) = \hat{i}_\infty(H)$ is an integer. If a non-degenerate fixed point $x_0$ satisfies $\hat{i}_H(x_0) = \hat{i}_\infty(H)$, then $i_H(x_0) = i_\infty(H)$. By \eqref{totalfloer}, the total Floer homology is one-dimensional, so only one fixed point can generate $HF_{i_\infty(H)}(H)$. Any additional fixed points with the same index must be canceled in the Floer complex. This requires at least one non-degenerate fixed point $z'$ with $i_H(z') = i_\infty(H) \pm 1$. Since $|i_H(z) - \hat{i}_H(z)| < 1$ for all non-degenerate $z$, we have $\hat{i}_H(z') \neq \hat{i}_\infty(H)$. Theorem \ref{theo1} then yields the desired conclusion.

If the eigenvalues of $\phi_Q^1$ are $e^{\pm i\theta} \in \mathrm{U} \setminus \mathbb{R}$. Here, $\hat{i}_\infty(H)$ is non-integral and $i_\infty(H)$ is the odd integer closest to it. For a non-degenerate fixed point $x_0$ with $\hat{i}_H(x_0) = \hat{i}_\infty(H)$, we have $i_H(x_0) = i_\infty(H)$. Again, \eqref{totalfloer} implies the Floer homology is one-dimensional, so there must be a fixed point $z'$ with $i_H(z') = i_\infty(H) \pm 1$ to cancel excess generators. Note that $i_H(z')$ is even in this case. By the index relationship, $\hat{i}_H(z') = i_H(z')$, so $\hat{i}_H(z')$ is an integer and hence distinct from $\hat{i}_\infty(H)$. Theorem \ref{theo1} again gives the conclusion.	

For $n = 2$, the relationship between the mean index and Conley–Zehnder index of a non-degenerate fixed point $z$ of $\phi_H^1$ falls into one of three cases:
\begin{enumerate}
	\item $\hat{i}_H(z) = i_H(z)$;
	\item $\hat{i}_H(z) = i_H(z) + 1 - \dfrac{\theta}{\pi}$, where $\theta \in (0,\pi) \cup (\pi, 2\pi)$;
	\item $\hat{i}_H(z) = i_H(z) + 2 - \dfrac{\theta_1}{\pi} - \dfrac{\theta_2}{\pi}$, where $\theta_1, \theta_2 \in (0,\pi) \cup (\pi, 2\pi)$.
\end{enumerate}
In the first two cases, $i_H(z)$ may be even or odd. In the third case, $i_H(z)$ must be even. The mean index $\hat{i}_H(z)$ is non-integral in the second case, and may be integral or non-integral in the third.

 If all eigenvalues of $\phi^1_Q$ are either entirely positive and not equal to $1$, entirely negative,  or form a quadruple $\{\rho \omega, \rho \overline{\omega}, \rho^{-1}\omega, \rho^{-1}\overline{\omega}\} \subset \mathbb{C} \setminus (\mathrm{U} \cup \mathbb{R})$,
then $i_\infty(H) = \hat{i}_\infty(H)$ and $i_\infty(H)$ is even. For any non-degenerate fixed point $x_0$ with $\hat{i}_H(x_0) = \hat{i}_\infty(H)$, we have $i_H(x_0) = i_\infty(H)$. By the structure of the total Floer homology \eqref{totalfloer}, there exists at least one non-degenerate fixed point $z'$ such that $i_H(z') = i_\infty(H) \pm 1$ , and $i_H(z')$ is odd. Index analysis implies that $|i_H(z') - \hat{i}_H(z')| < 1$, it follows that $\hat{i}_H(z') \neq \hat{i}_\infty(H)$. The desired conclusion then follows from Theorem \ref{theo1}.
\end{myproof}

\appendix
\section{Normal forms of symplectic matrices}
\label{Normal forms of symplectic matrices}

\textbf{Normal forms for eigenvalues $\pm 1$:}  
For $\lambda \in \mathbb{R}$ and $b \in \mathbb{R}$, define
\begin{equation}
	N_1(\lambda,b) = \begin{pmatrix}
		\lambda & b \\ 0 & \lambda    
	\end{pmatrix}.
\end{equation}
For $m \ge 2$, define $N_m(\lambda,b) \in \mathrm{Sp}(2m)$ by
\begin{equation}
	N_m(\lambda,b) = \begin{pmatrix}
		A_m(\lambda) & B_m(\lambda,b) \\
		0 & C_m(\lambda) \\
	\end{pmatrix},
\end{equation}
where $A_m(\lambda)$ is an $m \times m$ Jordan block for eigenvalue $\lambda$:
\begin{equation}\label{eq:Am}
	A_m(\lambda) = \begin{pmatrix} 
		\lambda & 1 & 0 & \cdots & 0 & 0\\
		0 & \lambda & 1 & \cdots & 0 & 0\\
		0 & 0 & \lambda & \cdots & 0 & 0 \\
		\vdots & \vdots & \vdots & \ddots & \vdots & \vdots \\
		0 & 0 & 0 & \cdots & \lambda & 1 \\
		0 & 0 & 0 & \cdots &  0 & \lambda \\
	\end{pmatrix},  
\end{equation}
$C_m(\lambda)$ is an $m \times m$ lower triangular matrix:
\begin{equation}\label{eq:Cm}
	C_m(\lambda) = \begin{pmatrix}
		-(-\lambda)^{-1} & 0 & 0 & \cdots & 0 \\
		-(-\lambda)^{-2} & -(-\lambda)^{-1} & 0 & \cdots & 0 \\
		-(-\lambda)^{-3} & -(-\lambda)^{-2} & -(-\lambda)^{-1} & \cdots & 0 \\
		\vdots & \vdots & \vdots & \ddots & \vdots \\
		-(-\lambda)^{-m} & -(-\lambda)^{-(m-1)} & -(-\lambda)^{-(m-2)} & \cdots & -(-\lambda)^{-1} \\
	\end{pmatrix},  
\end{equation}
and $B_m(\lambda,b)$ is an $m \times m$ lower triangular matrix parameterized by $b = (b_1,\dots,b_m) \in \mathbb{R}^m$:
\begin{equation}
	B_m(\lambda,b) = \begin{pmatrix}
		b_1 & 0 & 0 & \cdots & 0 \\
		b_2 & (-\lambda)b_2 & 0 & \cdots & 0 \\
		b_3 & (-\lambda)b_3 & (-\lambda)^2 b_3 & \cdots & 0 \\
		\vdots & \vdots & \vdots & \ddots & \vdots \\
		b_m & (-\lambda)b_m & (-\lambda)^2 b_m & \cdots & (-\lambda)^{m-1} b_m \\
	\end{pmatrix}.
\end{equation}
The normal forms for eigenvalues $\pm 1$ are:
\begin{equation}
	N_1(\pm 1,b) \quad (b = \pm 1, 0) \qquad \text{or} \qquad N_m(\pm 1,b) \quad (m \ge 2).
\end{equation}

\textbf{Normal forms for eigenvalues in $\mathrm{U} \setminus \mathbb{R}$:}
Fix $\omega = e^{i\theta} \in \mathrm{U} \setminus \mathbb{R}$ with $-\pi < \theta < \pi$, and let $\hat{\omega} = e^{i\hat{\theta}}$ where $\hat{\theta} = \theta$ or $-\theta$. Define
\begin{equation} 
	N_1(\hat{\omega},0) = R(\hat{\theta}) = \begin{pmatrix}
		\cos(\hat{\theta}) & -\sin(\hat{\theta}) \\
		\sin(\hat{\theta}) & \cos(\hat{\theta})
	\end{pmatrix},
\end{equation}
and for $m \ge 1$:
\begin{equation}\label{eq:N2m}
	N_{2m}(\hat{\omega},b) = \begin{pmatrix}
		A_{2m}(\hat{\omega}) & B_{2m}(b) \\ 
		0 & C_{2m}(\hat{\omega})    
	\end{pmatrix}.
\end{equation}
Here $A_{2m}(\hat{\omega})$ is a $2m \times 2m$ block Jordan form:
\begin{equation} \label{eq:A2m}
	A_{2m}(\hat{\omega}) = \begin{pmatrix}
		R(\hat{\theta}) & I_2 & 0 & \cdots & 0 \\
		0 & R(\hat{\theta}) & I_2 & \cdots & 0 \\
		0 & 0 & R(\hat{\theta}) & \cdots & 0 \\
		\vdots & \vdots & \vdots & \ddots & \vdots \\
		0 & 0 & 0 & \cdots & R(\hat{\theta}) \\
	\end{pmatrix},
\end{equation}
and $C_{2m}(\hat{\omega})$ is a block lower triangular matrix:
\begin{equation} \label{eq:C2m}
	C_{2m}(\hat{\theta}) = \begin{pmatrix}
		R(\hat{\theta}) & 0 & \cdots & 0 \\
		-R(2\hat{\theta}) & R(\hat{\theta}) & \cdots & 0 \\
		\vdots & \vdots & \ddots & \vdots \\
		(-1)^{m+1} R(m\hat{\theta}) & (-1)^m R((m-1)\hat{\theta}) & \cdots & R(\hat{\theta}) \\
	\end{pmatrix}.
\end{equation}
The symplectic condition implies $B_{2m}^T C_{2m}(\hat{\omega}) = C_{2m}(\hat{\omega})^T B_{2m}(b)$, where $B_{2m}(b)$ is a $2m \times 2m$ block matrix with $2 \times 2$ blocks $b_{i,j}$ satisfying $b_{i,j} = 0$ for $j > i + 1$.

For odd dimensions, define for $m \ge 1$:
\begin{equation}\label{eq:N2m+1}
	N_{2m+1}(\hat{\omega},b) = \begin{pmatrix}
		A & D & B & E \\
		0 & \cos(\hat{\theta}) & F^T & -\sin(\hat{\theta}) \\
		0 & 0 & C & 0 \\
		0 & \sin(\hat{\theta}) & G^T & \cos(\hat{\theta})
	\end{pmatrix},
\end{equation}
where $A$ and $C$ are as in \eqref{eq:A2m} and \eqref{eq:C2m}, and $D,E,F,G$ are $2m \times 1$ matrices determined by:
\begin{align}
	\hat{\theta} = \theta, &\quad D = (0,\dots,0,1,0)^T, \quad E = (0,\dots,0,0,1)^T \quad \text{if } b_{m+1} = -1, \\
	\hat{\theta} = -\theta, &\quad D = (0,\dots,0,0,1)^T, \quad E = (0,\dots,0,1,0)^T \quad \text{if } b_{m+1} = 1.
\end{align}
The complete list of normal forms for eigenvalues in $\mathrm{U} \setminus \mathbb{R}$ is:
\begin{equation}
	R(\hat{\theta}), \quad N_{2m}(\hat{\omega},b), \quad N_{2m+1}(\hat{\omega},b) \quad (m \ge 1).
\end{equation}

\textbf{Normal forms for eigenvalues outside $\mathrm{U}$:} 
For $\rho \in \mathbb{R}^+ \setminus \{0, 1\}$ and $\omega = e^{i\theta} \in \mathrm{U} \setminus \mathbb{R}$, define for $m \ge 1$:
\begin{equation}
	N_{2m}(\rho,\theta) = \begin{pmatrix}
		A_{2m}(\rho,\theta) & 0 \\
		0 & C_{2m}(\rho,\theta) \\
	\end{pmatrix},
\end{equation}
where $A_{2m}(\rho,\theta)$ is a $2m \times 2m$ block Jordan form for eigenvalues $\{\rho\omega, \rho\overline{\omega}\}$:
\begin{equation}
	A_{2m}(\rho,\theta) = \begin{pmatrix}
		\rho R(\theta) & I_2 & 0 & \cdots & 0 \\
		0 & \rho R(\theta) & I_2 & \cdots & 0 \\
		0 & 0 & \rho R(\theta) & \cdots & 0 \\
		\vdots & \vdots & \vdots & \ddots & \vdots \\
		0 & 0 & 0 & \cdots & \rho R(\theta) \\
	\end{pmatrix},  
\end{equation}
and $C_{2m}(\rho,\theta)$ is a block lower triangular matrix for eigenvalues $\{\rho^{-1}\omega, \rho^{-1}\overline{\omega}\}$:
\begin{equation}
	C_{2m}(\rho,\theta) = \begin{pmatrix}
		-(-\rho^{-1})R(\theta) & 0 & \cdots & 0 \\
		-(-\rho^{-1})^2 R(2\theta) & -(-\rho^{-1})R(\theta) & \cdots & 0 \\
		\vdots & \vdots & \ddots & \vdots \\
		-(-\rho^{-1})^m R(m\theta) & -(-\rho^{-1})^{m-1} R((m-1)\theta) & \cdots & -(-\rho^{-1})R(\theta) \\
	\end{pmatrix}.
\end{equation}Then the normal forms of symplectic matrices having the eigenvalue quadruple $\{\rho \omega, \rho \overline{\omega}, \rho^{-1}\omega, \rho^{-1}\overline{\omega}\}\subset   \bC \setminus (\Un \cup \bR)$:
$$N_{2m}(\rho,\theta)\quad m\ge 1.$$
For real eigenvalue pairs $\{\lambda, \lambda^{-1}\} \subset \mathbb{R} \setminus \{0, \pm 1\}$, the normal forms are:
\begin{equation}
	M_m(\lambda) = \begin{pmatrix} 
		A_m(\lambda) & 0 \\ 
		0 & C_m(\lambda)    
	\end{pmatrix} \quad (m \ge 1),
\end{equation}
with $A_m(\lambda)$ and $C_m(\lambda)$ as defined in \eqref{eq:Am} and \eqref{eq:Cm}.

\section{The precise expressions of certain logarithms}
\label{The precise expressions of certain logarithms}

For $M = N_1(-1, b)$, then $M = -e^{\hat{m}}$ with
\[
\hat{m} = \begin{pmatrix} 0 & b \\ 0 & 0 \end{pmatrix}.
\]

For $M=N_m  (-1,b)$, $m\ge2$, it can be known from  \cite{AH2022} that the logarithm has a series expansion
\begin{equation}  \label{eq:ln}
	\log(X)=\sum \limits_{k=1}^{\infty}\frac{(-1)^{k-1}}{k}(X-I)^k.	
\end{equation}
Since the upper left block matrix and the lower right block matrix of $-M-I$ are both nilpotent matrices, $-M-I$ is also nilpotent. Therefore, the right side of \eqref{eq:ln} is a finite summation. 
Thus, there exists $\hat{m}$ such that $M=-e^{\hat{m}}$.

For $M = R(\theta)$, we have $M = e^{\hat{m}}$ with
\[
\hat{m} = \begin{pmatrix}
	0 & -\theta \\
	\theta & 0
\end{pmatrix}.
\]

For $M = M_m(\lambda)$ with $\lambda \in \mathbb{R}^+ \setminus \{0, 1\}$, the identity $A_m(\lambda) C_m(\lambda)^T = I$ implies that $\hat{m}$ takes the block form:
\begin{equation} \label{eq:m}
	\hat{m} = \begin{pmatrix}
		\hat{m}_1 & 0 \\
		0 & -\hat{m}_1^T
	\end{pmatrix},
\end{equation}
where
\[
\hat{m}_1 = \begin{pmatrix}
	\log \lambda & \lambda^{-1} & -\frac{1}{2\lambda^2} & \frac{1}{3\lambda^3} & \cdots & \frac{(-1)^m}{(m-1)\lambda^{m-1}} \\
	0 & \log \lambda & \lambda^{-1} & -\frac{1}{2\lambda^2} & \cdots & \frac{(-1)^{m-1}}{(m-2)\lambda^{m-2}} \\
	0 & 0 & \log \lambda & \lambda^{-1} & \cdots & \frac{(-1)^{m-2}}{(m-3)\lambda^{m-3}} \\
	\vdots & \vdots & \vdots & \vdots & \ddots & \vdots \\
	0 & 0 & 0 & 0 & \cdots & \log \lambda
\end{pmatrix}.
\]
Then $M = e^{\hat{m}}$.

For $M = M_m(-\lambda)$ with $\lambda \in \mathbb{R}^+ \setminus \{0, 1\}$, the matrix $\hat{m}$ again takes the form \eqref{eq:m}, with
\[
\hat{m}_1 = \begin{pmatrix}
	\log \lambda & -\lambda^{-1} & -\frac{1}{2\lambda^2} & -\frac{1}{3\lambda^3} & \cdots & -\frac{1}{(m-1)\lambda^{m-1}} \\
	0 & \log \lambda & -\lambda^{-1} & -\frac{1}{2\lambda^2} & \cdots & -\frac{1}{(m-2)\lambda^{m-2}} \\
	0 & 0 & \log \lambda & -\lambda^{-1} & \cdots & -\frac{1}{(m-3)\lambda^{m-3}} \\
	\vdots & \vdots & \vdots & \vdots & \ddots & \vdots \\
	0 & 0 & 0 & 0 & \cdots & \log \lambda
\end{pmatrix},
\]
and $M = -e^{\hat{m}}$.

For $M = N_{2m}(\rho, \theta)$ with $m \ge 1$, the matrix $\hat{m}$ also takes the form \eqref{eq:m}, where
\[
\hat{m}_1 = \begin{pmatrix}
	\log(\rho R(\theta)) & (\rho R(\theta))^{-1} & -\frac{1}{2}(\rho R(\theta))^{-2} & \cdots & \frac{(-1)^m}{(m-1)}(\rho R(\theta))^{-(m-1)} \\
	0 & \log(\rho R(\theta)) & (\rho R(\theta))^{-1} & \cdots & \frac{(-1)^{m-1}}{(m-2)}(\rho R(\theta))^{-(m-2)} \\
	0 & 0 & \log(\rho R(\theta)) & \cdots & \frac{(-1)^{m-2}}{(m-3)}(\rho R(\theta))^{-(m-3)} \\
	\vdots & \vdots & \vdots & \ddots & \vdots \\
	0 & 0 & 0 & \cdots & \log(\rho R(\theta))
\end{pmatrix},
\]
with $\log(\rho R(\theta)) = \begin{pmatrix}
	\log \rho & -\theta \\
	\theta & \log \rho
\end{pmatrix}$, and $M = e^{\hat{m}}$.

\printbibliography

@article{FU2017,
  title={A higher dimensional Poincar{\'e}--Birkhoff theorem for Hamiltonian flows},
  author={Fonda, Alessandro and Ure{\~n}a, Antonio J},
  journal={Annales de l'Institut Henri Poincar{\'e} C},
  volume={34},                           
  number={3},
  pages={679--698},
  year={2017},
doi = {10.1016/j.anihpc.2016.04.002}
}

@article{BAME2023,
  title={Planar Hamiltonian systems: Index theory and applications to the existence of subharmonics},
  author={Boscaggin, Alberto and Mu{\~n}oz-Hern{\'a}ndez, Eduardo},
  journal={Nonlinear Analysis},
  volume={226},
  pages={113142},
  year={2023},
  publisher={Elsevier},
doi={10.1016/j.na.2022.113142 }
}

@article{GB2014, title={Periodic orbits of Hamiltonian systems linear and hyperbolic at infinity}, volume={271},  DOI={10.2140/pjm.2014.271.159}, number={1}, journal={Pacific Journal of Mathematics}, publisher={Mathematical Sciences Publishers}, author={Gürel, Başak}, year={2014}, month=sep, pages={159–182} }

@article{ML2024,
  title = {A {P}oincar\'{e}-{B}irkhoff {T}heorem for {A}symptotically {U}nitary {H}amiltonian {D}iffeomorphisms},
  author = {Masci, Leonardo},
   journal={arXiv preprint},
  doi = {10.48550/arXiv.2403.01855},
year={2024}
}

@article{MAKO2022,
  title={A generalized Poincar{\'e}--Birkhoff theorem},
  author={Moreno, Agustin and van Koert, Otto},
  journal={Journal of Fixed Point Theory and Applications},
  volume={24},
  number={2},
  pages={32},
  year={2022},
  publisher={Springer},
doi={10.1007/s11784-022-00957-6}
}

@book{Abb2001, title={Morse Theory for Hamiltonian Systems},DOI={10.1201/9781482285741}, publisher={Chapman and Hall/CRC}, address = {New York},author={Abbondandolo, Alberto}, year={2001}, month=mar }

@book{HHZE1995, title={Symplectic Invariants and Hamiltonian Dynamics},  DOI={10.1007/978-3-0348-8540-9}, publisher={Birkhäuser Basel}, address = {Basel} ,author={Hofer, Helmut and Zehnder, Eduard}, year={1994} }

@book{LY2012, title={Index Theory for Symplectic Paths with Applications}, DOI={10.1007/978-3-0348-8175-3}, address = {Basel} ,publisher={Birkhäuser Basel}, author={Long, Yiming}, year={2002} }

@book{AH2022, title={Mathematical Basics of Motion and Deformation in Computer Graphics}, DOI={10.1007/978-3-031-79561-9}, journal={Synthesis Lectures on Computer Graphics and Animation}, publisher={Springer International Publishing}, address = {Cham, Switzerland},author={Anjyo, Ken and Ochiai, Hiroyuki}, year={2017} }

@article{HNJ2006,
  title={Functions of matrices},
  author={Higham, Nicholas J},
  journal={Hogben [Hog07], page},
  year={2006}, publisher = {Springer},
  address = {Berlin, Heidelberg}
}

@article{BNRC1974, title={Normal forms for real linear Hamiltonian systems with purely imaginary eigenvalues}, volume={8},DOI={10.1007/bf01227796}, number={4}, journal={Celestial Mechanics}, publisher={Springer Science and Business Media LLC}, author={Burgoyne, N. and Cushman, R.}, year={1974}, month=jan, pages={435–443} }

@article{AAJ2022,
  title={Symplectic homology of convex domains and Clarke’s duality},
  author={Abbondandolo, Alberto and Kang, Jungsoo},
  journal={Duke Mathematical Journal}, DOI={10.1215/00127094-2021-0025},
  volume={171},
  number={3},
  pages={739--830},
  year={2022},
  publisher={Duke University Press}
}

@book{AMDM2014, title={Morse Theory and Floer Homology}, DOI={10.1007/978-1-4471-5496-9}, journal={Universitext}, publisher={Springer London}, author={Audin, Michèle and Damian, Mihai}, year={2014} }

@article{GVL2007, title={Coisotropic intersections}, volume={140},DOI={10.1215/s0012-7094-07-14014-6}, number={1}, journal={Duke Mathematical Journal}, publisher={Duke University Press}, author={Ginzburg, Viktor L.}, year={2007}, month=oct }

@article{GVL2010, title={The Conley conjecture}, volume={172},  DOI={10.4007/annals.2010.172.1127}, number={2}, journal={Annals of Mathematics}, publisher={Annals of Mathematics}, author={Ginzburg, Viktor L.}, year={2010}, pages={1127} }

@article{GJ2012,
  title={The Conley-Zehnder index for a path of symplectic matrices},
  author={Gutt, Jean},
  journal={arXiv preprint :1201.3728},
  year={2012}
}

@article{BPLP2003, title={Propagation in Hamiltonian dynamics and relative symplectic homology}, volume={119}, DOI={10.1215/s0012-7094-03-11913-4}, number={1}, journal={Duke Mathematical Journal}, publisher={Duke University Press}, author={Biran, Paul and Polterovich, Leonid and Salamon, Dietmar}, year={2003}, month=jul }

@article{FAHH1994, title={Symplectic homology I open sets in $\mathbb{C}^n$}, volume={215},  DOI={10.1007/bf02571699}, number={1}, journal={Mathematische Zeitschrift}, publisher={Springer Science and Business Media LLC}, author={Floer, A. and Hofer, H.}, year={1994}, month=jan }

@article{SM2000, title={On the action spectrum for closed symplectically aspherical manifolds}, volume={193},  DOI={10.2140/pjm.2000.193.419}, number={2}, journal={Pacific Journal of Mathematics}, publisher={Mathematical Sciences Publishers}, author={Schwarz, Matthias}, year={2000}, month=apr, pages={419–461} }

@article{GVBZ2010, title={Local Floer homology and the action gap}, volume={8},  DOI={10.4310/jsg.2010.v8.n3.a4}, number={3}, journal={Journal of Symplectic Geometry}, publisher={International Press of Boston}, author={Ginzburg, Viktor L. and Gürel, Başak Z.}, year={2010}, pages={323–357} }

@article{BHP2001, title={The Difference Between Consecutive Primes, II},  DOI={10.1112/plms/83.3.532}, number={3}, journal={Proceedings of the London Mathematical Society}, publisher={Wiley}, author={Baker, R. C. and Harman, G. and Pintz, J.}, year={2001}, month=nov, pages={532–562} }

@article{SZ1992, title={Morse theory for periodic solutions of hamiltonian systems and the maslov index}, volume={45}, DOI={10.1002/cpa.3160451004}, number={10}, journal={Communications on Pure and Applied Mathematics}, publisher={Wiley}, author={Salamon, Dietmar and Zehnder, Eduard}, year={1992}, month=dec, pages={1303–1360} }

@article{SDM1990, title={Morse Theory, the Conley Index and Floer Homology}, volume={22},  DOI={10.1112/blms/22.2.113}, number={2}, journal={Bulletin of the London Mathematical Society}, publisher={Wiley}, author={Salamon, Dietmar}, year={1990}, month=mar, pages={113–140} }

@article{SD1997,
  title={Lectures on Floer homology},
  author={Salamon, Dietmar},
  journal={Symplectic geometry and topology (Park City, UT, 1997)},
  volume={7},
  pages={143--229},
  year={1999}
}

@article{FA11989, title={Witten’s complex and infinite-dimensional Morse theory}, volume={30},   DOI={10.4310/jdg/1214443291}, number={1}, journal={Journal of Differential Geometry}, publisher={International Press of Boston}, author={Floer, Andreas}, year={1989}, month=jan }

@article{FA21989, title={Symplectic fixed points and holomorphic spheres}, volume={120}, DOI={10.1007/bf01260388}, number={4}, journal={Communications In Mathematical Physics}, publisher={Springer Science and Business Media LLC}, author={Floer, Andreas}, year={1989}, month=dec, pages={575–611} }

@book{AV2013, title={Mathematical Methods of Classical Mechanics},DOI={10.1007/978-1-4757-2063-1},   journal={Graduate Texts in Mathematics},  publisher = {Springer},
  address = {New York}, author={Arnold, V. I.}, year={1989} }

@article{CKM1999, title={A Unified Approach to Linear and Nonlinear Normal Forms for Hamiltonian Systems}, volume={27}, DOI={10.1006/jsco.1998.0244}, number={1}, journal={Journal of Symbolic Computation}, publisher={Elsevier BV}, author={Churchill, R.C. and Kummer, M. }, year={1999}, month=jan, pages={49–131} }

@article{BAOR2014, title={Subharmonic solutions of the forced pendulum equation: a symplectic approach}, DOI={10.1007/s00013-014-0644-2}, number={5}, journal={Archiv der Mathematik}, publisher={Springer Science and Business Media LLC}, author={Boscaggin, A. and Ortega, R. and Zanolin, F.}, year={2014}, month=may, pages={459–468} }

@article{conley1984, title={Morse‐type index theory for flows and periodic solutions for Hamiltonian Equations}, volume={37},DOI={10.1002/cpa.3160370204}, number={2}, journal={Communications on Pure and Applied Mathematics}, publisher={Wiley}, author={Conley, Charles and Zehnder, Eduard}, year={1984}, month=mar, pages={207–253} }

\end{document}